\documentclass[12pt,a4paper]{article}
\usepackage[english]{babel}
\usepackage{fancyhdr}
\usepackage[latin1]{inputenc}
\usepackage{amsmath}
\usepackage{amsfonts}
\usepackage{amssymb}
\usepackage{graphicx}
\usepackage{latexsym}
\usepackage{amsthm}
\usepackage{color}
\usepackage{midpage}
\usepackage[width=15.00cm]{geometry}
\usepackage{amscd}
\usepackage{tikz-cd}
\usepackage{hyperref}

\numberwithin{equation}{section}
\theoremstyle{plain}
\newtheorem{theorem}{Theorem}[section]
\theoremstyle{corollary}
\newtheorem{corollary}[theorem]{Corollary}
\theoremstyle{lemma}
\newtheorem{lemma}[theorem]{Lemma}
\theoremstyle{proprosition}

\theoremstyle{assumption}

\theoremstyle{condition}

\theoremstyle{definition}
\newtheorem{definition}[theorem]{Definition}
\theoremstyle{example}

\newtheorem{remark}[theorem]{Remark}

\def\authors#1{{ \begin{center} #1 \vspace{0pt} \end{center} } \smallskip}
\def\institution#1{{\sl \begin{center} #1 \vspace{0pt} \end{center} } }
\def\inst#1{\unskip $^{#1}$}
\def\title#1{{\huge\bf  \begin{center} #1 \vspace{0pt} \end{center}  } \smallskip}

\begin{document}
	
	\title{\sc Spherical Poisson Waves}
	\authors{\large Solesne Bourguin\inst{*}, Claudio Durastanti\inst{\dagger}, \\Domenico Marinucci\inst{\ddagger} and Anna Paola Todino\inst{§}}
	\institution{\inst{*}Department of Mathematics and Statistics, Boston University\\
		\inst{\dagger}Dipartimento SBAI, Sapienza Universit\`a di Roma\\
		\inst{\ddagger}Dipartimento di Matematica, Universit\`a di Roma ``Tor Vergata" \\
	\inst{§} Dipartimento di Scienze Statistiche, Sapienza Universit\`a di Roma}



\begin{abstract}
We introduce a model of Poisson random waves in $\mathbb{S}^{2}$ and we
study Quantitative Central Limit Theorems when both the rate of the Poisson
process and the energy (i.e., frequency) of the waves (eigenfunctions)
diverge to infinity. We consider finite-dimensional distributions, harmonic
coefficients and convergence in law in functional spaces, and we investigate
carefully the interplay between the rate of divergence of eigenvalues and
Poisson governing measures.
\end{abstract}

\begin{itemize}
\item \textbf{Keywords and Phrases: }Random Spherical Eigenfunctions,
Poisson Random Fields, Quantitative Central Limit Theorems.

\item \textbf{AMS Classification: 60G60; 60F05, 60B10}.
\end{itemize}

\section{Introduction}
\subsection{Motivations}

The analysis of Gaussian eigenfunctions on different manifolds has recently
become a very attractive area of research - it started in the mathematics
literature mainly about a decade ago (\cite{nazarov,Wig}) and it has then
covered a number of different questions and circumstances, including the
Euclidean case (Berry's Random Wave Model, see \cite{Berry 1977,Berry
	2002,DalmaoEstradeLeon,npr,Vidotto21}), Random Spherical Harmonics
(eigenfunctions on the sphere, see \cite{CMW, CM2020,CM2018,CW, MRW, Todino20}),
Arithmetic Random Waves (eigenfunctions on the torus, see \cite%
{Cam19,CKW,KKW,Maff,MPRW2015,RW08}) and other manifolds (see \cite%
{CH,Duerickx,Sarnak}). The leading motivation for such a strong interest
comes mainly from the physical sciences, and in particular from an \emph{%
	ansatz} by Michael Berry in a 1977 paper \cite{Berry 1977}, where he claimed
that Gaussian random waves could be taken as a universal model to
approximate the behaviour even of deterministic eigenfunctions in the
high-energy limit (i.e., for diverging eigenvalues) under "generic" boundary
conditions.

A common argument to justify the universality of Gaussian behaviour for
eigenfunctions in the physics literature is the \emph{random phase model}
(see \cite{WigSurvey} and the references therein), which we can describe as
follows. Working on $\mathbb{R}^{2},$ assume we observe the superposition of
$N$ waves at a given frequency $k$, that is%
\begin{equation}
T_{k;N}(x)=\frac{1}{\sqrt{N}}\sum_{j=1}^{N}\exp (ik\left\langle \theta
_{j},x\right\rangle +\phi _{j})\text{ ,}  \label{BRW2}
\end{equation}%
for $x\in \mathbb{R}^{2}$, $k\in \mathbb{R}^{+},$ where $\left\{ \theta
_{j}\right\} _{j=1,...,N}$ are random directions on the unit circle and  $%
\left\{ \phi _{j}\right\} _{j=1,...,N}$ are random phases. By a standard
Central Limit Theorem it is then immediate to show that $T_{k;N}(x)$
converges in distribution to a zero mean Gaussian field $\widetilde{T}_{k}(%
\mathbf{\cdot })$ with covariance function given by%
\begin{equation*}
\mathbb{E}\left[ \widetilde{T}_{k}(x_{1})\widetilde{T}_{k}(x_{2})\right]
=J_{0}(k\left\Vert x_{1}-x_{2}\right\Vert _{2}),
\end{equation*}%
where $J_{0}\left( \cdot \right) $ is the Bessel function of order $0$,
given by
\begin{equation*}
J_{0}(u)=\sum_{m=0}^{\infty }(-1)^{m}\frac{u^{2m}}{2^{2m}(m!)^{2}}\text{ }.
\end{equation*}%
\

For a fixed value of the wavelength parameter, hence, the validity of a
Central Limit Theorem result follows from very standard arguments.

It should be noticed, however, that the literature on random eigenfunctions
has actually been developed under the implicit framework of a double
asymptotic setting. Indeed, on the one hand, a diverging number of random phases
is taken to ensure that the behaviour of random eigenfunctions is Gaussian;
on the other hand, Gaussianity is taken for granted when investigating the
asymptotic behaviour of random eigenfunctions in the high-frequency/high
energy sense (i.e., for diverging eigenvalues). Some natural questions are
hence the following - given that Gaussianity has been established for a
\emph{fixed} eigenvalue $k,$ can we justify the use of this assumption in
the limit as $k\rightarrow \infty ?$ Can we allow at the same time the
eigenvalues to grow together with the number of random phases, and still
have a Central Limit Theorem? Do we need some conditions that relate
of the
divergence for the eigenvalue $k$ to the rate of divergence of the number of
random phases $N?$ How many ``\emph{random phases}'' do we
need, in the language of Berry's celebrated model, in order for the Gaussian
approximation to hold at high frequencies?

In this paper we try to address these questions in the case of random
eigenfuctions defined on the two-dimensional sphere $\mathbb{S}^{2};$ the
choice of the sphere is motivated by the fact that it represents the most
interesting case from the point of view of physical applications and it is
known to exhibit the same covariance structure as the Euclidean case, in the
scaling limit (due to so-called Hilb's asymptotics, see \cite[Equation 8.21.7]{Szego}, and \cite{Wig}). The extension of these results to the planar case
does not seem to pose any conceptual difficulties; it would be more
interesting to explore this setting in the case of Arithmetic Random Waves,
which is known to exhibit some differences with respect to Euclidean and
Spherical circumstances. We leave this extension for further research.

\subsection{The model}

Equation (\ref{BRW2}) shows that waves with random phases can actually be
viewed as a superposition of deterministic eigenfunctions centred on random
locations $\left\{ y_{j}\right\} _{j=1,...,N}$. To achieve an analogous
construction in the spherical case, we need to recall first that the
Laplacian operator in $\mathbb{S}^{2}$ is defined by%
\begin{equation*}
\Delta _{\mathbb{S}^{2}}:=\frac{1}{\sin \theta }\frac{\partial }{\partial
	\theta }\sin \theta \frac{\partial }{\partial \theta }+\frac{1}{\sin
	^{2}\theta }\frac{\partial ^{2}}{\partial \varphi ^{2}}\text{ ;}
\end{equation*}%
in the spherical case, a deterministic eigenfunction centred on $y\in
\mathbb{S}^{2}$ can be constructed by%
\begin{equation*}
e_{\ell ;y}(\cdot):\mathbb{S}^{2}\rightarrow \mathbb{R}\text{ , }e_{\ell
	;y}(\cdot):=\sqrt{\frac{2\ell +1}{4\pi }}P_{\ell }(\left\langle
\cdot,y\right\rangle )\text{ , }
\end{equation*}%
where we have introduced the family of Legendre polynomials%
\begin{equation*}
P_{\ell }(t):=\frac{1}{2^{\ell }\ell !}\frac{d^{\ell }}{dt^{\ell }}%
(t^{2}-1)^{\ell }\text{ , }\ell =0,1,2,\ldots;\text{ }t\in \lbrack 0,1]\text{
	.}
\end{equation*}%
The choice of normalization ensures that $P_{\ell }(1)\equiv 1$ for all $%
\ell $ and moreover%
\begin{equation}  \label{eq:norm1}
\left\Vert e_{\ell ;y}\right\Vert _{L^{2}(\mathbb{S}^{2})}=\int_{\mathbb{S}%
	^{2}}\frac{2\ell +1}{4\pi }P_{\ell }^{2}(\left\langle x,y\right\rangle
)dy=P_{\ell }(\left\langle x,x\right\rangle )=1\text{ ,}
\end{equation}%
in view of the duplication formula, see for instance, \cite[Section 13.1.2]%
{MP}; also, we have that $\left\{ e_{\ell ;y}(\cdot)\right\} $ satisfies the
Helmholtz equation%
\begin{equation*}
\Delta _{\mathbb{S}^{2}}e_{\ell ;y}(x)+\lambda _{\ell }e_{\ell ;y}(x)=0\text{
	, }\ell =0,1,2,...,
\end{equation*}%
where $-\lambda _{\ell }=-\ell (\ell +1)$ is the sequence of eigenvalues of
the spherical Laplacian, see again \cite{MP,Wig}.

We convey the idea of random phases on the sphere by introducing a
superposition of waves centred on Poisson distributed random points on $%
\mathbb{S}^{2}.$ Here is a more formal setting. For a more rigorous
definition of Poisson random measures, the reader is referred to Section \ref%
{sec:generalresults}.

\begin{definition}
	The Poisson spherical random wave model (with rate $\nu _{t})$ is defined by
	\begin{equation*}
	T_{\ell ;t}(x):=\frac{1}{\sqrt{\nu _{t}}}\int_{\mathbb{S}^{2}}\sqrt{\frac{%
			2\ell +1}{4\pi }}P_{\ell }(\left\langle x,\xi \right\rangle )dN_{t}(\xi ),
	\end{equation*}%
	where $\left\{ N_{t}(\cdot )\right\} $ is a Poisson process on the sphere
	with governing intensity measure
	\begin{equation*}
	\mathbb{E}\left[ N_{t}(A)\right] =\nu _{t}\times \mu (A)\text{ for all }A\in
	\mathcal{B(}\mathbb{S}^{2})\text{ ,}
	\end{equation*}
	where $\mu$ is the Lebesgue measure on $\mathbb{S}^2$, see for further
	details Section \ref{sec:generalresults} below.
\end{definition}

Our model implies that for all measurable sets $A\subset \mathbb{S}^{2}$ and
$t\geq 0,$ $N_{t}(A)$ is a Poisson random variable with expected value equal
to $\nu _{t}\times \mu (A),$ and for $A_{1}\cap A_{2}=\varnothing $, $N_{t}(A_{1})$ and $N_{t}(A_{2})$ are independent. We can also write the
Poisson spherical wave as
\begin{equation*}
T_{\ell ;t}(x)=\frac{1}{\sqrt{\nu _{t}}}\sum_{k=1}^{N_{t}(\mathbb{S}^{2})}%
\sqrt{\frac{2\ell +1}{4\pi }}P_{\ell }(\left\langle x,\xi _{k}\right\rangle )%
\text{ ,}
\end{equation*}%
so that we can view spherical Poisson random waves as occuring from the sum
of a (random) number of deterministic waves, centred at points which are
uniformly distributed on the sphere.

It is now convenient to introduce the standard basis for the $(2\ell +1)$%
-dimensional space of eigenfunctions corresponding to the eigenvalue $%
\lambda _{\ell };$ the elements of the basis are sometimes called \emph{%
	fully normalized spherical harmonics, }and are defined as the normalized
eigenfunctions $\left\{ Y_{\ell m}\right\} _{m=-\ell ,...,\ell }$ which
satisfy the further condition (in spherical coordinates)
\begin{equation*}
Y_{\ell m}:\mathbb{S}^{2}\rightarrow \mathbb{R}\text{ , }\frac{\partial ^{2}%
}{\partial \varphi ^{2}}Y_{\ell m}(\theta ,\varphi )=-m^{2}Y_{\ell m}(\theta
,\varphi )\text{ .}
\end{equation*}%
The elements of the real fully normalized spherical harmonics basis can be
written explicitly as the normalized product of the so-called Legendre
associated function $P_{\ell}^{m}: \left[-1, 1\right] \mapsto \mathbb{R}$ of
degree $\ell$ and order $m$, which depends only on $\theta$ and is defined
by
\begin{equation*}
P_{\ell}^{ m}(t):=(1-t^{2})^{m/2}\frac{d^{m}}{dt^{m}}P_{\ell }(t)\text{ , }%
t\in \lbrack 0,1]\text{ }
\end{equation*}%
(see \cite[Equation 13.7]{MP}), and a trigonometric function depending only
on $\phi$, that is,
\begin{equation*}
Y_{\ell m}\left(\theta, \phi\right) =
\begin{cases}
\sqrt{\frac{2\ell+1}{2\pi}\frac{\left(\ell-m\right)!}{\left(\ell+m\right)!}}
P_{\ell}^{ m} \left( \cos \theta \right) \cos \left( m \phi \right) & \text{%
	for } m \in \left\{1, \ldots, \ell\right\} \\
\sqrt{\frac{2\ell+1}{4\pi}} P_{\ell}\left(\cos \theta\right) & \text{for } m
=0 \\
\sqrt{\frac{2\ell+1}{2\pi}\frac{\left(\ell+m\right)!}{\left(\ell-m\right)!}}
P_{\ell}^{ -m}\left(\cos \theta\right) \sin \left( -m \phi \right) & \text{%
	for } m \in \left\{-\ell, \ldots, -1\right\}%
\end{cases}%
,
\end{equation*}
see, for example, \cite[Remark 3.37]{MP}.

It should be noted, however, that none of the results below depend on the
specific choice of our basis; they would hold unaltered for any orthonormal
system. The most important properties of the fully normalized spherical
harmonics are the addition and duplication formula (see respectively \cite[%
Eq. (3.42) and Sec. 13.1.2]{MP}), which are given respectively by%
\begin{equation}
\sum_{m=-\ell }^{\ell }Y_{\ell m}(x){Y}_{\ell m}(y)=\frac{2\ell +1}{4\pi }%
P_{\ell }(\left\langle x,y\right\rangle )\text{ ,}  \label{addition}
\end{equation}%
\begin{equation}
\int_{\mathbb{S}^{2}}\frac{2\ell +1}{4\pi }P_{\ell }(\left\langle
x,z\right\rangle )\frac{2\ell +1}{4\pi }P_{\ell }(\left\langle
z,y\right\rangle )dz=\frac{2\ell +1}{4\pi }P_{\ell }(\left\langle
x,y\right\rangle )\text{ ,}  \label{duplication}
\end{equation}%
for all $x,y\in \mathbb{S}^{2}.$ Using the addition formula yields
\begin{equation*}
T_{\ell ;t}(x)=\frac{1}{\sqrt{\nu _{t}}}\sqrt{\frac{4\pi }{2\ell +1}}%
\sum_{k=1}^{N_{t}(\mathbb{S}^{2})}\sum_{m=-\ell }^{\ell }Y_{\ell
	m}(x)Y_{\ell m}(\xi _{k})=\sum_{m=-\ell }^{\ell }\hat{a}_{\ell, m}(t)Y_{\ell
	m}(x),
\end{equation*}%
where the random spherical harmonic coefficients $\left\{ \hat{a}_{\ell
	,	m}\right\} _{m=-\ell ,...,\ell }$ are defined by%
\begin{equation*}
\hat{a}_{\ell ,m}(t):=\sqrt{\frac{4\pi }{(2\ell +1)\nu _{t}}}\sum_{k=1}^{N_{t}}{Y%
}_{\ell m}(\xi _{k})\text{ },
\end{equation*}%
where $\{\xi_k\}$ are the points charged by the Poisson process. Note that%
\begin{equation*}
\mathbb{E}[\hat{a}_{\ell ,m}(t)\hat{a}_{\ell ^{\prime },m^{\prime }}(t)]=\delta
_{m}^{m^{\prime }}\delta _{\ell }^{\ell ^{\prime }}\frac{4\pi }{(2\ell +1)}
\end{equation*}%
and
\begin{equation*}
\mathbb{E}[T_{\ell ;t}(x)T_{\ell ;t}(y)]=P_{\ell }(\langle x,y\rangle )\text{
	.}
\end{equation*}%
It is also easy to verify that the Parseval's identity holds, i.e.
\begin{equation*}
||T_{\ell ;t}||_{L^{2}(\mathbb{S}^{2})}^{2}=\int_{\mathbb{S}^{2}}T_{\ell
	;t}^{2}(x)dx=\sum_{m=-\ell }^{\ell }|\hat{a}_{\ell, m}(t)|^{2}.
\end{equation*}

\subsection{Overview of the main results}

In this work, we consider the convergence in law of the Poisson random
spherical eigenfunctions to a Gaussian limit when both the rate of the
Poisson process and the eigenvalue sequence $\lambda _{\ell }$ diverge to
infinity; see for instance \cite{Eichelsbacher,LSY,Last} and the references
therein for some recent results on quantitative convergence bounds in a
Poisson framework. We focus on three different cases:

a) We study the convergence of the finite-dimensional distributions for a
fixed array of $d$-points $(x_{1},x_{2},...,x_{d})\in \mathbb{S}^{2},$ with
special emphasis on the univariate marginal distribution for $d=1;$ here we
prove that a quantitative Central Limit Theorem holds insofar we have that $%
d^{2}\sqrt{\log \ell }=o(\sqrt{\nu _{t}})$. In particular, for the special
case $d=1$ asymptotic Gaussianity holds for eigenvalues that increase
polynomially fast with respect to the rate of occurrence of Poisson events.

b) We also study the convergence in law for the vector of spherical harmonic
coefficients $\left\{ \hat{a}_{\ell, \cdot }\right\} _{m=-\ell ,...,\ell };$
again a multivariate Central Limit Theorem would be straightforward, but
here we provide a quantitative version when $\ell $ (and hence the dimension
of the vector itself) grows with $\nu _{t}.$ The bound we obtain here is of
order $\sqrt{\frac{\log \ell }{\nu _{t}}},$ thus entailing that multivariate
asymptotic Gaussianity holds provided $\sqrt{\log \ell }=o(\sqrt{\nu _{t}})$%
. Out of this bound, it is also possible to derive an alternative rate of
convergence for finite-dimensional distributions of order $d,$ which turns
out to be $d\sqrt{\ell \log \ell /\nu _{t}}$, see Remark \ref{clementi}. For fixed $d$, this is clearly worse than the bound we discussed in the previous point, but it can actually be better if one envisages $d$ as growing with $\ell$ at a suitably fast rate.

c) We then consider functional convergence results, where we view the
eigenfunctions $\left\{ T_{\ell ;t}\right\} $ as random elements $T_{\ell
	;t}:\Omega \rightarrow L^{2}(\mathbb{S}^{2})$, i.e. as measurable
applications with the topology induced on $L^{2}(\mathbb{S}^{2})$ by the
standard metric
\begin{equation*}
d^{2}(f,g):=\left\Vert f-g\right\Vert _{L^{2}(\mathbb{S}^{2})}^{2}=\int_{%
	\mathbb{S}^{2}}|f(x)-g(x)|^{2}dx\text{ .}
\end{equation*}%
Exploiting some very recent and important results by \cite{BCD} (see also
\cite{BC}), we are able here to show that a quantitative Central Limit
Theorem holds under the simple condition that $\nu _{t}\rightarrow \infty .$
This is apparently surprising, because in this functional case it turns out
that asymptotic Gaussianity will hold no matter how fast the sequence of
eigenvalues diverge to infinity, on the contrary of what we have stated for
the (apparently simpler) cases under b) and c). A careful inspection of the
results reveals that the apparent paradox is due to the topological
structure induced by the $\left\Vert .\right\Vert _{L^{2}(\mathbb{S}^{2})},$
which is much coarser than the one given for instance, by the sup
norm. In particular, weak convergence with respect to $\left\Vert
.\right\Vert _{L^{2}(\mathbb{S}^{2})}$ does not entail convergence of the
finite-dimensional distributions, not even univariate ones.

d) Finally, we establish a quantitative Central Limit Theorem in functional spaces which induce finer topologies; we focus in particular on Sobolev spaces (see Theorem \ref{Sobolev}).
Here, we are able to obtain the rate $\frac{\sqrt{\pi} \left(1+\sqrt{\ell(\ell+1)}\right)^{2\alpha}}{2\sqrt{\nu
		_{t}}} + \frac{2\pi \left(1+\sqrt{\ell(\ell+1)}\right)^{3\alpha}}{\sqrt{\nu _{t}}}$, which is much worse than in both the $L^2$ case and for marginal distributions. However, functional convergence in Sobolev spaces with sufficient regularity is, of course, a much stronger result; in particular, among others, it does imply convergence of the finite-dimensional distributions at fixed locations on the sphere, as detailed below in Section \ref{2.3.2} , Corollary \ref{final}.
\begin{remark}
	It is well-known (see \cite[Corollary 1]{BTA}) that convergence in $%
	L^{2}(\cdot )$ does indeed entail pointwise convergence in the case of
	Reproducing Kernel Hilbert Spaces (RKHS). Because the space of spherical
	eigenfunctions is indeed a RKHS, again the point in c) may sound
	counterintuitive. There is a subtle point here, as $\ell $ increases, we are
	actually dealing with a sequence of RKHS; whereas it is indeed possible to
	bound the pointwise norm with the $L^{2}$ distance \textsl{up to a constant,}
	the \textquotedblleft constant\textquotedblright\ does vary with $\ell ,$
	and indeed it diverges to infinity as we shall discuss below; so no
	contradiction arises.
\end{remark}

\subsection{Some remarks on the nature of asymptotics}

At this stage, it is important to add some remarks on the nature of our asymptotic results. We note first that \emph{both the multipole index $\ell $ and the
	Poisson rate diverge jointly to infinity in our framework}; for fixed multipoles $\ell$, convergence to Gaussianity remains true but becomes rather trivial and uninnteresting. As a consequence of this double asymptotics framework, the covariance functions of the processes that we study do not converge to the covariance of a well-defined, measurable function on the sphere. 

Indeed it is easy to see that for any choice of fixed points on the sphere the covariance of our process converges to zero as $\ell \rightarrow \infty .$ This implies that the
limiting process (if it existed) would not be mean square continuous; but such a process cannot be well-defined (i.e., measurable) as proved in \cite{MP13}.

However, this apparent difficulty allows us to exploit the full power of quantitative central limit results. Indeed, this class of theorems does	not require, by any means, that the sequences converge to a well defined
limiting distribution. One can have two sequences of random variables $%
X_{n},Y_{n}$ and show that $d_{W}(X_{n},Y_{n})\rightarrow 0$ as $%
n\rightarrow \infty ,$ meaning that we can approximate the distribution of $%
X_{n}$ arbitrarily well with the distribution of $Y_{n},$ for $n$ large
enough, independently from the fact that $X_{n}$ converges or not to a limit distribution. For instance,$Y_{n}$ could be a sequence of
Gaussian variables with oscillating mean and variance $\mu _{Y_{n}}$ and $%
\sigma _{Y_{n}}^{2},$ and still one could use the Gaussian quantiles to
approximate the distribution of $X_{n}$ as $N(\mu _{Y_{n}},\sigma
_{Y_{n}}^{2}).$ Very much the same can be said below, where  the covariance
operators do not converge to meaningful limits. For ideas that are in a
broad sense related, one could also think about the large $p,$ large $n$
framework in random matrix theory (see for instance \cite{BD,BDT} where it
is shown that under appropriate conditions, the laws of large Wishart
random matrices become indistinguishable from the laws of the Gaussian orthogonal ensemble).

For completeness, we add that it may be possible to get
some form of nondegenerate limiting behaviour for random waves: in particular, if neglecting the spherical structure and focussing only on shrinking domains
around a single fixed point $x\in \mathbb{S}^2$, then it could be possible to show that the
scaling limit of the waves $T_{\ell }$ when projected on the tangent plane
converges locally to random eigenfunctions on $\mathbb{R}^{2}$ (Berry's
random waves). The price to pay for this approach would however be high: the result would no longer deal with convergence on the sphere, which is what we are studying in this work. Moreover, this approach would not allow to answer the question that we addressed here and that we consider interesting for physical applications:  given a random spherical harmonic $%
T_{\ell }$ with $\ell $ suitably large, what is the order of magnitude of the governing Poisson rate that is required  for the Gaussian approximation to be adequate? This is exactly the issue that we address in the sequel, under a variety of different circumstances.

\subsection{Acknowledgements}

The research by SB was supported in part by the Simons Foundation grant 635136. The research by CD was supported by Sapienza Grant  RM12117A6212F538. The
research by DM was partially supported by the MIUR Departments of Excellence
Programs Math@ToV and MatModTov.



\section{Main Results}

Before we proceed with the statement of our results, we need to recall
briefly the probability metrics that we are going to exploit, which are
defined by


a) Wasserstein metric: for any two random vectors $X,Y:\Omega \rightarrow
\mathbb{R}^{d}$%
\begin{equation*}
d_{W}(X,Y)=\sup_{h\in Lip(1)}|\mathbb{E}\left[ h(X)\right] -\mathbb{E}\left[
h(Y)\right] |\text{ , }
\end{equation*}%
\begin{equation*}
h\in Lip(1)\Leftrightarrow h:\mathbb{R}^{d}\mathcal{\rightarrow }\mathbb{R}%
:M_{1}(h)\leq 1\text{ .}
\end{equation*}%
where $M_{1}(h)$ is defined by
\begin{equation}
M_{1}(h):=\sup_{\substack{ x,y\in \mathbb{R}^{d},  \\ x\neq y}}\frac{%
	|h(x)-h(y)|}{||x-y||_{ \mathbb{R}^{d}}}.  \label{eq:MLC}
\end{equation}%
b) $d_{3}$ metric: for any two random vectors $X,Y:\Omega \rightarrow
\mathbb{R}^{d}$ such that $\mathbb{E}||X||_{\mathbb{R}^{d}}^{2}$, $\mathbb{E}%
||Y||_{\mathbb{R}^{d}}^{2}<\infty $,
\begin{equation*}
d_{3}(X,Y)=\sup_{h\in \mathcal{I}}|\mathbb{E}\left[ h(X)\right] -\mathbb{E}%
\left[ h(Y)\right] |
\end{equation*}%
where $\mathcal{I}$ indicates the collection of all functions $h\in \mathbb{C%
}^{3}(\mathbb{R}^{d})$ such that $||h^{\prime \prime }||_{\infty }\leq 1$
and $||h^{\prime \prime \prime }||_{\infty }\leq 1$.

c) Functional $d_3$ metric: for a general function space $K$ we have that
$C_b^3(K)$ is the class of real-valued functions on $K$ that have bounded Fr\'echet derivatives up to order three. This space is equipped with the norm 
\[ ||h||_{C_b^3(K)}= \sup_{j=1,2,3} \sup_{x \in K} ||D^j h(x)||_{K^{\otimes j}} .\]
Then, given a Hilbert space $K$ and
any two random elements $X,Y:\Omega \rightarrow K$%
\begin{eqnarray*}
	d_{3}(X,Y) &=&\sup_{h\in C_b^3(K)}|\mathbb{E}\left[ h(X)\right] -\mathbb{E}%
	\left[ h(Y)\right] |\text{ . }
\end{eqnarray*}
Clearly b) can be viewed as a special case of c), for $\mathcal{H}=
\mathbb{R}^d.$ We refer to \cite[Appendix C]{noupebook} for more
discussion and examples on probability metrics and their mutual
relationships.

We divide our results below in three subsections, referring respectively to
finite-dimensional distributions, harmonic coefficients and functional
convergence.

\subsection{Convergence of the finite dimensional distributions}

We start from a simple univariate case; this is of course implied by the $d$%
-dimensional result that we give below, but we prefer to treat it on its own
for clarity of exposition and to optimize the value of the relevant
constants.

\begin{theorem}
	\label{one-dim} (One-dimensional case) Let the notation above prevail and $%
	Z\sim \mathcal{N}(0,1)$. For all $x\in \mathbb{S}^{2}$ we have that%
	\begin{equation*}
	d_{W}(T_{\ell ;t}(x),Z)\leq \left(\frac{\sqrt{3}}{2\pi^2} +\frac{2}{3}\sqrt{%
		\frac{3}{2\pi^3 }} \right) \sqrt{\frac{\log \ell }{\nu _{t}}}+o\left(\sqrt{%
		\frac{\log \ell }{\nu _{t}}}\right)\text{ }.
	\end{equation*}
\end{theorem}

\begin{proof} Note first that, because $T_{\ell ;t}\left( x\right) =\frac{1%
	}{\sqrt{\nu _{t}}}\int_{\mathbb{S}^{2}}\sqrt{\frac{2\ell +1}{4\pi }}P_{\ell
	}\left( \left\langle x,\xi \right\rangle \right) dN_{t}(\xi)$ we are in the
	domain of validity of Fourth Moment Theorems for integral functionals of
	Poisson processes, see for instance \cite{PZ} and many subsequent papers. In
	particular, we shall exploit \cite[Theorem 1.7]{DVZ}, which we recall in
	Section \ref{sec:generalresults} below, see Theorem \ref{FMT-DP}. To apply
	this result, we need to compute the fourth moment of $T_{\ell ; t},$ which is
	given by
	
	\begin{equation*}
	\mathbb{E}[T_{\ell;t }^{4}(x)]=\sum_{m_{1}m_{2}m_{3}m_{4}}\mathbb{E}[\hat{a}%
	_{\ell ,m_{1}}(t)\ \hat{a}_{\ell, m_{2}}(t)\hat{a}_{\ell, m_{3}}(t)\ \hat{a}_{\ell,
		m_{4}}(t)]Y_{\ell m_{1}}(x){Y}_{\ell m_{2}}(x)Y_{\ell m_{3}}(x){Y}_{\ell
		m_{4}}(x)\text{ ;}
	\end{equation*}%
	substituting the value of $\hat{a}_{\ell ,m}$ we have that
	
	\begin{equation*}
	\mathbb{E}[T_{\ell ;t}^{4}(x)]=\left( \frac{4\pi }{\nu _{t}(2\ell +1)}\right)
	^{2}\sum_{m_{1},...m_{4}=-\ell }^{\ell }\sum_{k_{1},...,k_{4}=1}^{N_{t}(\mathbb{S}^2)}%
	\mathbb{E}[\hat{Y}_{\ell m_{1}}(\xi _{k_{1}})\hat{Y}_{\ell m_{2}}(\xi
	_{k_{2}})\hat{{Y}}_{\ell m_{3}}(\xi _{k_{3}})\hat{Y}_{\ell m_{4}}(\xi
	_{k_{4}})]
	\end{equation*}
	
	\begin{equation*}
	\times Y_{\ell m_{1}}(x){Y}_{\ell m_{2}}(x)Y_{\ell m_{3}}(x){Y}_{\ell
		m_{4}}(x)\text{ }.
	\end{equation*}%
	Exploiting the addition formula (\ref{addition}) we get
	\begin{eqnarray*}
		\mathbb{E}[T_{\ell;t }^{4}(x)] &=&\left( \frac{2\ell +1}{4\pi }\right) ^{2}\frac{1%
		}{\nu _{t}^{2}}\mathbb{E}\left[\sum_{k_{1},...,k_{4}=1}^{N_{t}(\mathbb{S}^2)}P_{\ell }(\langle
		\xi _{k_{1}},x\rangle )P_{\ell }(\langle \xi _{k_{2}},x\rangle )P_{\ell
		}(\langle \xi _{k_{3}},x\rangle )P_{\ell }(\langle \xi _{k_{4}},x\rangle )\right]
		\\
		&=&\left( \frac{2\ell +1}{4\pi }\right) ^{2}\frac{1}{\nu _{t}^{2}}\left( \nu
		_{t}\mathbb{E}\left[ P_{\ell }(\langle \xi _{k_{1}},x\rangle )^{4}\right]
		+3\nu _{t}^{2}\mathbb{E}\left[ P_{\ell }(\langle \xi _{k_{1}},x\rangle )^{2}%
		\right] ^{2}\right) .
	\end{eqnarray*}%
	In \cite[Lemma 2.3]{M e W 2011}, it has been shown that
	\begin{equation*}
	\int_{0}^{1}P_{\ell }^{4}(t)\,dt\sim \frac{3}{2\pi ^{2}}\frac{\log \ell }{%
		\ell ^{2}}\text{ ,}
	\end{equation*}%
	where for any two positive sequences $\left\{ a_{\ell },b_{\ell }\right\}
	_{\ell =1,2,\ldots }$ we write
	\begin{equation*}
	a_{\ell }\sim b_{\ell }\Leftrightarrow \lim_{\ell \rightarrow \infty }\frac{%
		a_{\ell }}{b_{\ell }}=1\text{ }.
	\end{equation*}%
	Thus we get
	\begin{equation*}
	\mathbb{E}\left[ P_{\ell }(\langle \xi _{k_{1}},x\rangle )^{4}\right] =\int_{%
		\mathbb{S}^{2}}P_{\ell }(\langle z,x\rangle )^{4}dz\sim 4\pi \frac{3}{2\pi
		^{2}}\frac{\log \ell }{\ell ^{2}}\text{ , as }\ell \rightarrow \infty \text{
	}.
	\end{equation*}%
	Moreover, since
	\begin{equation*}
	\int_{0}^{1}P_{\ell }(t)^{2}\,dt=\frac{1}{2\ell +1},
	\end{equation*}%
	we also have that
	\begin{equation*}
	\mathbb{E}[P_{\ell }(\langle \xi _{k_{1}},x\rangle )^{2}]=\int_{\mathbb{S}%
		^{2}}P_{\ell }(\langle z,x\rangle )^{2}dz=\frac{4\pi }{2\ell +1}.
	\end{equation*}%
	It follows that
	\begin{equation*}
	\mathbb{E}[T_{\ell ;t}^{4}(x)]=3+\frac{3}{2\pi ^{3}}\frac{\log \ell }{\nu _{t}}%
	+o\left( \frac{\log \ell }{\nu _{t}}\right) \text{ }.
	\end{equation*}%
	Applying Theorem \ref{FMT-DP}, the thesis of the theorem follows.
	
\end{proof}

\begin{remark}
	(\emph{A Comparison with Needlet/Wavelet Coefficients}) We note here that
	the constraint on the rate of convergence of the eigenvalues with respect to
	the rate in the Poisson governing intensity measure is very weak; Theorem %
	\ref{one-dim} shows that asymptotic Gaussianity will continue to hold even
	if we allow $\lambda _{\ell }$ to grow as any polynomial function of the
	rate $\nu _{t}.$ This is in sharp contrast with what is observed in related
	circumstances for the behaviour of spherical wavelet/needlet coefficients
	(see, for example, \cite{DMP}). To compare those results with the ones
	presented here, we recall that needlet coefficients corresponding to $\xi
	\in \mathbb{S}^{2}$ in the notation of this can be considered as equivalent
	to (after normalization)%
	\begin{eqnarray*}
		\beta _{j}(\xi ) &:=&\sum_{\ell =2^{j-1}}^{2^{j+1}}b\left(\frac{\ell }{2^{j}}%
		\right)T_{\ell ;t}(\xi ) \\
		&=&\frac{1}{\sqrt{\nu _{t}}}\int_{\mathbb{S}^{2}}\psi _{j}(\left\langle
		x,\xi \right\rangle )dN_{t}(\xi )\text{ ,} \\
		\psi _{j}(\left\langle x,\xi \right\rangle ) &:=&\sum_{\ell
			=2^{j-1}}^{2^{j+1}}b\left(\frac{\ell }{2^{j}}\right)\frac{2\ell +1}{4\pi }%
		P_{\ell }(\left\langle x,\xi \right\rangle )\text{ ,}
	\end{eqnarray*}%
	where $\left\{ b(\frac{\ell }{2^{j}})\right\} _{\ell =2^{j-1},...,2^{j+1}}$
	is a sequence of suitably constructed weights (see \cite{NPW, BKMP}),
	normalized here so that the coefficients have unit variance. It can be shown
	that (see \cite{DMP})
	\begin{equation*}
	d_{3}(\beta _{j}(\xi ),Z)=O\left(\sqrt{\frac{2^{2j}}{\nu _{t}}}\right)=O\left(\sqrt{\frac{%
			\ell _{j}^{2}}{\nu _{t}}}\right)\text{ , }\ell _{j}:=2^{j}\text{ ,}
	\end{equation*}%
	so that asymptotic Gaussianity follows only for multipoles which grow
	sub-linearly with respect to $\sqrt{\nu _{t}}.$ Heuristically, the kernel $\left\{
	\psi _{j}(\left\langle \xi ,.\right\rangle )\right\} $ is characterized by a
	very fast decay, as opposed to Legendre polynomials (see \cite{NPW,BKMP});
	its support can be considered to shrink as $\ell _{j}^{-2},$ and hence the
	"effective" Poisson rate behaves as $\ell _{j}^{-2}\times \nu _{t}.$ This is
	very different from what we observe in this paper for Poisson random waves,
	because as we mentioned above the support of Legendre polynomials does not
	shrink in any similar way as $\ell $ grows, which makes asymptotic
	Gaussianity much simpler to achieve.
\end{remark}

In order to focus on the more general finite dimensional distributions case,
we need first to introduce some additional notation. Let us fix $d$ points $%
x_{1},x_{2},\dots ,x_{d}$ on $\mathbb{S}^{2}$ and introduce the random
vector
\begin{equation}
F_d=(T_{\ell ;t}(x_{1}),T_{\ell ;t}(x_{2}),\ldots ,T_{\ell ;t}(x_{d}))\text{ ;}
\label{F}
\end{equation}%
the elements of the covariance matrix of $F_{d}$, which we denote by $\Gamma_{d}:=\Gamma^{(\ell)}_{d}(F_d)$, are easily seen to be given by
\begin{equation*}
\Gamma _{d;ij}:=\mathbb{E}[T_{\ell;t }(x_{i})T_{\ell;t }(x_{j})]=P_{\ell
}(\langle x_{i},x_{j}\rangle )\text{ , }i,j=1,\dots ,d\text{ }.
\end{equation*}%
Note that the elements on the diagonal $\Gamma _{d;ii}$, $i=1,\ldots
,d$, are exactly equal to 1 (cf. Eq. \eqref{eq:norm1}).

Before we state our next result, some words on notation. Similarly to what
was done for the definition of Wasserstein distance, where the supremum was
taken with respect to Lipschitz functions with constant no larger than one,
we might have defined the $d_{3}(\cdot,\cdot)$ distance with respect to a
more definite class of functions, such that the two factors $%
M_{2}(g),M_{3}(g)$ are smaller than one (say). These constants are
explicitly given by
\begin{equation*}
M_{k}(g):=\sup_{x\neq y}\frac{||D^{k-1}g(x)-D^{k-1}g(y)||_{op}}{
	||x-y||_{\mathbb{R} ^{d}}}, \quad k\in \mathbb{N}, \quad g
\in C^{k-1}\left(\mathbb{R} ^{d}\right),
\end{equation*}%
where $D^{k-1}g(x)$ is the $(k-1)$--th derivative of $g$ at any point $x\in%
\mathbb{R}^d$ (see also Section \ref{sec:generalresults}). We also recall
that for a vector $x=(x_1,\dots,x_d)^T \in \mathbb{R}^d$, we denote by $%
||x||_{2}$ its Euclidean norm and for a matrix $A \in \mathbb{R}^{d\times d}$%
, we denote by $||A||_{op}$ the operator norm induced by the Euclidean norm,
i.e.,
\begin{equation*}
||A||_{op}:=\sup\{ ||Ax||_{2} : ||x||_2=1\}.
\end{equation*}

We prefer however the current formulation which is more general and
flexible, although slightly more cumbersome. We write $Z_{d}$ for a Gaussian
vector of dimension $d$ with zero mean and covariance matrix equal to $%
\Gamma _{d}$.

\begin{theorem}
	\label{multifdd} We have that
	\begin{equation*}
	d_{3}(F_d,Z_{d})\leq \sup_{g\in C^{3}}B_{3}(g;d)\sqrt{\frac{3}{2\pi ^{3}}}d%
	\sqrt{\frac{\log \ell }{\nu _{t}}}+o\left( d^{2}\sqrt{\frac{\log \ell }{\nu
			_{t}}}\right)
	\end{equation*}%
	where
	\begin{equation*}
	B_{3}(g;d):=\frac{\sqrt{2d}}{4}M_{2}(g)+\frac{2d}{9}M_{3}(g)\text{ }.
	\end{equation*}
\end{theorem}

\begin{proof} First of all we note that all the components of $F_d$ belong
	to the same first-order Poisson Wiener chaos and then we can apply Theorem %
	\ref{FMT-DVZ}. Moreover, from Theorem \ref{one-dim} we have that
	\begin{equation*}
	\mathbb{E}[T_{\ell;t }^{4}]=3+\frac{3}{2\pi ^{3}}\frac{\log \ell }{\nu _{t}}%
	+o\left( \frac{\log \ell }{\nu _{t}}\right) \text{ }.
	\end{equation*}%
	Hence we obtain
	\begin{eqnarray*}
		|\mathbb{E}[g(F_d)-\mathbb{E}[g(Z_{d})]| &\leq
		&B_{3}(g;d)\sum_{i=1}^{d}(\text{cum}_{4}(F_{d;i}))^{1/2} \\
		&\sim &B_{3}(g;d)\sqrt{\frac{3}{2\pi ^{3}}}\sum_{i=1}^{d}\left( \frac{\log
			\ell }{ \nu _{t}}\right) ^{1/2}
	\end{eqnarray*}%
	where $F_{d;i}$ is the $i$-th component of the vector $F_d$ and 
	
	\begin{equation*}
	B_{3}(g;d)=\frac{\sqrt{2d}}{4}M_{2}(g)+\frac{2\sqrt{d\text{Tr}(\Gamma _{d})}}{9}%
	M_{3}(g).
	\end{equation*}%
	We recall that for a zero mean random variable $F$, the fourth-cumulant $\text{cum}_4(F)$ is given by $$\text{cum}_{4}(F)=\mathbb{E}[F^4]-3(\mathbb{E}[F^2])^2,$$ see for instance \cite{noupebook} for more discussions and details.
	Noting that $\text{Tr}(\Gamma _{d})=d$, the theorem is proved.
\end{proof}

\begin{remark}
	Note that we have established the bound%
	\begin{equation*}
	d_{3}(F,Z_{d})=O\left( d^{2}\sqrt{\frac{\log \ell }{\nu _{t}}}\right) \text{
		,}
	\end{equation*}%
	which holds when the dimension $d$ grows with $\ell $ and $\nu _{t}.$
\end{remark}

\subsection{Convergence of spherical harmonic coefficients}

Let us consider the vector
\begin{equation*}
V_{\ell;t}:=(\hat{a}_{\ell ,-\ell }\left( t\right) ,\dots ,\hat{a}%
_{\ell ,\ell }(t))=\{\hat{a}_{\ell ,m}(t)\}_{m=-\ell
	,\dots ,\ell }\text{ ,}
\end{equation*}
where%
\begin{equation*}
\hat{a}_{\ell ,m}\left(t\right) =\sqrt{\frac{4\pi }{(2\ell +1)}}%
\frac{1}{\sqrt{\nu _{t}}}\sum_{k=1}^{N_{t}}{Y}_{\ell ,m}(\xi _{k})\text{ }.
\end{equation*}%
Observe that each entry of $V_{\ell;t}$ is built by evaluating a different
element of the fully normalized spherical harmonic basis $\left\{ Y_{\ell
	m}\right\} $ over the same set of random points $\left\{ \xi _{k}\right\} $.
As a consequence, the random coefficients are neither independent nor
identically distributed, although they are still uncorrelated; indeed we
have that, for all $m,m^{\prime }=-\ell ,\ldots ,\ell $%
\begin{eqnarray}
\mathbb{E}\left[ \hat{a}_{\ell ,m}\left( t\right) \right]
&=&\int_{\mathbb{S}^{2}}Y_{\ell m}(z)dz=0\text{ }; \nonumber \\
\mathbb{E}\left[ \hat{a}_{\ell ,m}\left( t\right) {\hat{a}}_{\ell
	,m^{\prime }}\left( t\right) \right] &=&\frac{4\pi }{(2\ell +1)}%
\int_{\mathbb{S}^{2}}Y_{\ell m}(z){Y}_{\ell m^{\prime }}(z)dz \nonumber \\
&=&\delta _{m}^{m^{\prime }}\frac{4\pi }{(2\ell +1)}. \label{cov-a}
\end{eqnarray}

\begin{theorem}\label{multidim}
	\label{convV} Let $Z_{2\ell +1}$ be a Gaussian vector of dimension $2\ell +1$
	with zero mean and diagonal variance/covariance matrix equal to $\frac{4\pi}{2\ell +1}I_{2\ell+1}$. Then we have that
	\begin{equation*}
	d_{3}(V_{\ell ;t},Z_{2\ell +1})\leq \sup_{g\in C^{3}}B_{3}(g;\ell )\sqrt{8%
		\sqrt{4\pi }\frac{1.539\log \ell }{\ell \nu _{t}}+O\left( \frac{1%
		}{\ell \nu _{t}}\right) }\text{ ,}
	\end{equation*}%
	where
	\begin{equation*}
	B_{3}(g;\ell ):=\frac{\sqrt{2(2\ell +1)}}{4}M_{2}(g)+\frac{2}{9}\sqrt{(2\ell
		+1)4\pi }M_{3}(g).
	\end{equation*}
\end{theorem}

It should be noted that the resulting bound is of order $\sqrt{\frac{\log
		\ell }{\nu _{t}}}.$ Before proving the theorem we need some lemmas.

\begin{lemma}\label{lemma}
	We have that
	\begin{equation*}
	\mathbb{E}[\hat{a}_{\ell ,m}^{4}(t)]=\left( \frac{4\pi }{2\ell +1}\right) ^{2}\frac{1%
	}{\nu _{t}}\mathbb{E}\left[ |Y_{\ell m}(\xi _{1})|^{4}\right] +3\left(
	\frac{4\pi }{2\ell +1}\right) ^{2} \text{ .}
	\end{equation*}
\end{lemma}

\begin{remark}\label{remark}
	From \cite{M e W 2012}, p. 23 we have:%
	\begin{equation*}
	\mathbb{E}\left[ Y_{\ell m}(\xi _{1})^{4}\right] =\frac{(2\ell +1)}{\sqrt{%
			4\pi }}\left[ \sum_{L}\frac{\left( C_{\ell ,0;\ell ,0}^{L,0}\right)
		^{2}\left( C_{\ell ,-m;\ell ,m}^{L,0}\right) ^{2}}{2L+1}\right] \text{ ,}
	\end{equation*}%
	where $\{C_{\ell_1 ,m_1;\ell_2, m_2}^{\ell_3, m_3}\}$ are the Clebsch-Gordan
	coefficients, defined in Appendix \ref{C-G coefficients} (see \cite%
	{Varshalovich} Chapter 8 and \cite{MP} p.77 cap.3.5). \\Note that as $\ell
	\rightarrow \infty $ we have (see p. 16 \cite{M e W 2012})
	
	\begin{equation}\label{Pl4}
	\sum_{L}\frac{\left( C_{\ell ,0;\ell ,0}^{L,0}\right) ^{4}}{2L+1}=
	\int_{0}^{1} P_\ell(t)^4\,dt \sim \frac{3}{2\pi^2} \frac{\log \ell }{\ell
		^{2}} ,
	\end{equation}
\end{remark}

\begin{corollary}
	\label{cum4a}We have that
	\begin{equation*}
	\text{cum}_{4}(\hat{a}_{\ell ,m}\left( t\right) )=\frac{4\pi
		\sqrt{4\pi }}{(2\ell +1)}\frac{1}{\nu _{t}}\left[ \sum_{L}\frac{\left(
		C_{\ell ,0;\ell ,0}^{L,0}\right) ^{2}\left( C_{\ell ,-m;\ell
			,m}^{L,0}\right) ^{2}}{2L+1}\right]
	\end{equation*}%
	and as $\ell \rightarrow \infty $%
	\begin{eqnarray*}
		\text{cum}_{4}(\hat{a}_{\ell ,0}\left( t\right) ) &=&\frac{4\pi
			\sqrt{4\pi }}{(2\ell +1)}\frac{1}{\nu _{t}}\left[ \sum_{L}\frac{\{C_{\ell
				,0;\ell ,0}^{L,0}\}^{4}}{2L+1}\right] \\
		&\sim &\frac{6}{\sqrt{\pi }}\frac{\log \ell }{\ell ^{3}\nu _{t}}\text{ .}
	\end{eqnarray*}
\end{corollary}

The first result of Corollary \ref{cum4a} follows by exploiting Remark \ref{remark} in Lemma \ref{lemma} and recalling the definition of the fourth cumulant of a zero mean random variable. The second result is due to (\ref{Pl4}).

We also need the following lemma.
\begin{lemma}
	\label{bound} We have that
	\begin{equation}  \label{eq:bound}
	\frac{4\pi \sqrt{4\pi }}{(2\ell +1)^{2}}\frac{1}{\nu _{t}}\leq \sum_{m=-\ell
	}^{\ell }\text{cum}_{4}(\hat{a}_{\ell ,m}(t))\leq 8\sqrt{4\pi }\frac{1.539
		\log \ell }{\ell (2\ell +1)\nu _{t}}+O\left( \frac{1}{\ell ^{2}\nu _{t}}%
	\right)
	\end{equation}
\end{lemma}

\begin{proof}[Proof of Lemma \ref{bound}] From Corollary \ref{cum4a} we get
	
	\begin{eqnarray*}
		\sum_{m=-\ell }^{\ell }\text{cum}_{4}(\hat{a}_{\ell, m}(t))&=&\sum_{m=-\ell }^{\ell }%
		\frac{4\pi \sqrt{4\pi }}{(2\ell +1)}\frac{1}{\nu _{t}}\bigg[\sum_{L}\frac{%
			\{C_{\ell ,0;\ell ,0}^{L,0}\}^{2}\{C_{\ell ,-m;\ell ,m}^{L,0}\}^{2}}{2L+1}%
		\bigg]\\
		&=&\frac{4\pi \sqrt{4\pi }}{(2\ell +1)}\frac{1}{\nu _{t}}\sum_{L}\frac{%
			\{C_{\ell ,0;\ell ,0}^{L;0}\}^{2}}{2L+1}\sum_{m=-\ell }^{\ell }\{C_{\ell
			,-m;\ell ,m}^{L,0}\}^{2}.
	\end{eqnarray*}%
	In view of the unitary property (\ref{unitary 1}) recalled in Appendix \ref%
	{C-G coefficients} (see for example eq. (3.62) \cite{MP}),
	we have
	\begin{equation*}
	\sum_{m=-\ell }^{\ell }\{C_{\ell ,-m;\ell ,m}^{L,0}\}^{2}=1
	\end{equation*}%
	and then we obtain
	\begin{equation}
	\sum_{m=-\ell }^{\ell }\text{cum}_{4}(\hat{a}_{\ell, m}(t))=\frac{4\pi \sqrt{4\pi }}{%
		(2\ell +1)}\frac{1}{\nu _{t}}\sum_{L}\frac{\{C_{\ell ,0;\ell ,0}^{L,0}\}^{2}%
	}{2L+1}.  \label{sumcum4a}
	\end{equation}%
	The relation (\ref{3.67 libro Domenico}) between the Clebsch-Gordan
	coefficients and the 3j Wigner coefficients (defined in Appendix \ref{C-G
		coefficients}), leads to
	\begin{equation*}
	\frac{\{C_{\ell ,0;\ell ,0}^{L,0}\}^{2}}{2L+1}=%
	\begin{pmatrix}
	\ell  & \ell  & L \\
	0 & 0 & 0%
	\end{pmatrix}%
	^{2}.
	\end{equation*}%
	%
	From Lemma A.1 in \cite{M e W 2011} we have that
	
	\begin{equation}
	\begin{pmatrix}
	\ell & \ell & L \\
	0 & 0 & 0%
	\end{pmatrix}%
	^{2}= \frac{2}{\pi }\gamma _{\ell L}\frac{1}{L(2\ell -L)^{1/2}(2\ell
		+L)^{1/2}};  \label{boundWigner3j}
	\end{equation}%
	where $\left\{ \gamma _{\ell L}\right\} _{\ell =1,2,...}$ is a deterministic
	sequence such that $0.596\leq \gamma_{\ell L}\leq 1.539$ uniformly in $\ell$ and $L$. Then%
	\begin{equation*}
	\sum_{L=0}^{2\ell }\frac{\{C_{\ell, 0;\ell, 0}^{L,0}\}^{2}}{2L+1}=\{C_{\ell,
		0;\ell, 0}^{0,0}\}^{2}+\sum_{L=1}^{2\ell -1}\frac{\{C_{\ell, 0;\ell,
			0}^{L,0}\}^{2}}{2L+1}+\frac{\{C_{\ell, 0;\ell, 0}^{2\ell, 0}\}^{2}}{4\ell +1}%
	.
	\end{equation*}%
	In view of (\ref{1 pag 248}) and (\ref{rem8.9MP}), the first and the last
	term of this sum are easily seen to satisfy
	\begin{equation*}
	\{C_{\ell, 0;\ell, 0}^{0,0}\}^{2}=\frac{1}{2\ell +1}\text{ , }\frac{%
		\{C_{\ell, 0;\ell, 0}^{2\ell, 0}\}^{2}}{4\ell +1}\leq \frac{1}{(4\ell +1)^{2}%
	}\text{ .}
	\end{equation*}
	In view of (\ref{boundWigner3j}), we also get
	
	\begin{equation*}
	\sum_{L=1}^{2\ell -1}\frac{\{C_{\ell, 0;\ell, 0}^{L,0}\}^{2}}{2L+1}=
	\frac{2}{\pi }\sum_{L=1}^{2\ell -1}\frac{\gamma _{\ell L}}{L(2\ell
		-L)^{1/2}(2\ell +L)^{1/2}}\text{ .}
	\end{equation*}%
	Now we note that
	\begin{equation*}
	\sum_{L=1}^{2\ell -1}\frac{1}{L(2\ell -L)^{1/2}(2\ell +L)^{1/2}}\leq \frac{1%
	}{(2\ell +1)^{1/2}}\sum_{L=1}^{2\ell -1}\frac{1}{L(2\ell -L)^{1/2}}
	\end{equation*}%
	\begin{equation}  \label{bo}
	=\frac{1}{(2\ell +1)^{1/2}}\sum_{L=1}^{\ell }\frac{1}{L(2\ell -L)^{1/2}}+%
	\frac{1}{(2\ell +1)^{1/2}}\sum_{L=\ell +1}^{2\ell -1}\frac{1}{L(2\ell
		-L)^{1/2}}.
	\end{equation}
	The first sum can be bounded as follows:%
	\begin{equation*}
	\frac{1}{(2\ell +1)^{1/2}}\sum_{L=1}^{\ell }\frac{1}{L(2\ell -L)^{1/2}}\leq
	\frac{1}{(2\ell +1)^{1/2}(\ell )^{1/2}}\sum_{L=1}^{\ell }\frac{1}{L}\leq
	\frac{\log \ell }{\ell }.
	\end{equation*}%
	For the second sum in (\ref{bo}) we have that%
	\begin{equation*}
	\frac{1}{(2\ell +1)^{1/2}}\sum_{L=\ell +1}^{2\ell -1}\frac{1}{L(2\ell
		-L)^{1/2}}\leq \frac{1}{(2\ell +1)^{1/2}(\ell +1)}\sum_{L=\ell +1}^{2\ell -1}%
	\frac{1}{(2\ell -L)^{1/2}}
	\end{equation*}%
	and changing variable $L^{\prime }=2\ell -L$ we get%
	\begin{equation*}
	\frac{1}{(2\ell +1)^{1/2}(\ell +1)}\sum_{L^{\prime }=1}^{\ell -1}\frac{1}{%
		(L^{\prime 1/2})}\leq \frac{1}{(2\ell +1)^{1/2}(\ell +1)}(\ell )^{1/2}\leq
	\frac{1}{\ell }.
	\end{equation*}%
	Finally we have
	\begin{equation*}
	\sum_{L=0}^{2\ell }\frac{\{C_{\ell, 0;\ell, 0}^{L,0}\}^{2}}{2L+1}\leq \frac{2%
	}{\pi }\left( \frac{1.539\log \ell }{\ell }+\frac{1.539}{\ell }+\frac{\pi}{2(4\ell +1)^{2}} +\frac{\pi}{2(2\ell+1)}\right)
	\end{equation*}%
	and exploiting this bound in (\ref{sumcum4a}) we conclude that
	\begin{equation*}
	\sum_{m=-\ell }^{\ell }\text{cum}_{4}(\hat{a}_{\ell, m}(t))\leq 8\sqrt{4\pi }\frac{1.539\log \ell }{\ell (2\ell +1)\nu _{t}}+O\left( \frac{1}{\ell
		^{2}\nu _{t}}\right) .
	\end{equation*}%
	On the other hand it can be easily seen that
	\begin{eqnarray*}
		\sum_{m=-\ell }^{\ell }\text{cum}_{4}(\hat{a}_{\ell, m}(t))&=&\frac{4\pi \sqrt{4\pi }}{%
			(2\ell +1)}\frac{1}{\nu _{t}}\sum_{L}\frac{\{C_{\ell, 0;\ell, 0}^{L,0}\}^{2}%
		}{2L+1}\geq \frac{4\pi \sqrt{4\pi }}{(2\ell +1)}\frac{1}{\nu _{t}}\{C_{\ell
			,0;\ell, 0}^{0,0}\}^{2}\\
		&=&\frac{4\pi \sqrt{4\pi }}{(2\ell +1)}\frac{1}{\nu _{t}}\frac{1}{2\ell +1}=%
		\frac{4\pi \sqrt{4\pi }}{(2\ell +1)^{2}}\frac{1}{\nu _{t}}\text{ ,}
	\end{eqnarray*}%
	where we used (\ref{1 pag 248}) in the second-last equality, and that leads to the thesis of the lemma.
\end{proof}

\begin{proof}[Proof of Theorem \ref{convV}] We exploit the multidimensional Fourth Moment Theorem 
	in \cite{DVZ}, in particular Theorem \ref{FMT-DVZ} in Section \ref{appendix}, to get
	
	\begin{equation}
	|\mathbb{E}[g(V_{\ell;t})]-\mathbb{E}[g(Z_{2\ell+1})]|\leq B_{3}(g;\ell)\sum_{m=-\ell }^{\ell }%
	\sqrt{\text{cum}_{4}(\hat{a}_{\ell ,m}\left( t\right) )}.
	\end{equation}
	Applying the following Cauchy-Schwarz inequality
	\begin{equation}\label{CS}%
	\sum_{i=1}^{d}\sqrt{a_{i}}\leq d^{\frac{1}{2}}\left(
	\sum_{i=1}^{d}a_{i}\right) ^{\frac{1}{2}},
	\end{equation}%
	it follows that
	\begin{equation}
	|\mathbb{E}[g(V_{\ell;t})]-\mathbb{E}[g(Z_{2\ell+1})]|\leq B_{3}(g;\ell)\sqrt{2\ell +1}\sqrt{%
		\sum_{m=-\ell }^{\ell }\text{cum}_{4}(\hat{a}_{\ell, m}(t))}.  \label{eq}
	\end{equation}%
	In view of the definition of $B_{3}(g;d)$ in (\ref{B3}) and using (\ref{cov-a})
	we find
	\begin{equation*}
	B_{3}(g;\ell)=A_{2}(g;\ell)+\frac{2\sqrt{(2\ell +1)4\pi }}{9}M_{3}(g)\text{
		, }
	\end{equation*}
	with
	\begin{equation*}
	A_{2}(g;\ell)=\frac{\sqrt{2(2\ell +1)}}{4 }M_{2}(g)\text{ }.
	\end{equation*}
	Exploiting the upper bound of \eqref{eq:bound} in Lemma \ref{bound}, the
	thesis of the theorem follows.
\end{proof}

\begin{remark}\label{clementi}
	Since the covariance matrix of the vector $V_{\ell;t}$ is positive definite, we can also apply the second part of Theorem \ref{FMT-DVZ} in Section \ref{sec:generalresults} to get a quantitative Central Limit Theorem in terms of the metric $d_2$, defined as follows.
	For any two random vectors $X,Y: \Omega \to \mathbb{R}^d$ such that $\mathbb{E}[||X||^2_{\mathbb{R}^d}], \mathbb{E}[||Y||^2_{\mathbb{R}^d}]<\infty,$ we have that
	$$ d_2(X,Y) =\sup_{h \in \mathcal{I}} |\mathbb{E}[h(X)]- \mathbb{E}[h(Y)]|$$ where $\mathcal{I}$ indicates the collection of all functions $h \in \mathbb{C}^2(\mathbb{R}^d)$ such that $||h||_{Lip} \leq 1$ and $M_2(h)\leq 1$. \\
	Similarly to the proof of Theorem \ref{multidim}, we find
	$$d_2(V_{\ell;t}, Z_{2\ell+1}) \leq \sup_{g\in C^{2}} B_2(g;\ell) \sqrt{8\sqrt{4\pi} \frac{1.539\log \ell}{\ell \nu_t}+ O\left(\frac{1}{\ell \nu_t}\right)}$$
	where
	$$B_2(g;\ell)=A_1(g) + \frac{\sqrt{2\pi} }{6} \sqrt{2\ell+1} M_2(g) \mbox{ and } A_1(g)= \sqrt{\frac{2\ell+1}{4\pi}} \frac{1}{\sqrt{\pi}} M_1(g).$$
\end{remark}

A natural question concerns the relationship between the results of this
section and the quantitative bounds for the convergence of the
finite-dimensional distributions provided in the previous pages. To this
aim, we recall the definition of $F_{d}=( T_{\ell
	;t}(x_{1}),...,T_{\ell ;t}(x_{d})) $ and we simply note that%
\begin{eqnarray*}
	F_{d} &=&%
	\begin{pmatrix}
		Y_{\ell ,-\ell }(x_{1}) & Y_{\ell ,-\ell +1}(x_{1}) & \dots & Y_{\ell ,\ell
		}(x_{1}) \\
		\vdots & \vdots & \vdots & \vdots \\
		Y_{\ell ,-\ell }(x_{d}) & Y_{\ell ,-\ell +1}(x_{d}) & \dots & Y_{\ell ,\ell
		}(x_{d})%
	\end{pmatrix}%
	\begin{pmatrix}
		\widehat{a}_{\ell ,-\ell }(t) \\
		\vdots \\
		\widehat{a}_{\ell ,\ell }(t)%
	\end{pmatrix}
	\\
	&=:&\Psi_{\ell ;d}(a_{\ell ,\cdot })\text{ ,}
\end{eqnarray*}%
where $\Psi_{\ell ;d}:\mathbb{R}^{2\ell +1}\rightarrow \mathbb{R}^{d}$ is a linear function and hence obviously continuous (and bounded).
Indeed, because $\left\vert Y_{\ell ,m}\right\vert \leq \sqrt{\frac{2\ell +1%
	}{2\pi }}$ uniformly over the sphere, for all functions $h:\mathbb{R}^{d}\rightarrow \mathbb{R}$
we can write
\begin{equation*}
h(T_{\ell ;t}(x_{1}),...,T_{\ell ;t}(x_{d}))=h\circ \Psi_{\ell;d}(a_{\ell ,\cdot })=:\widetilde{h}(a_{\ell ,\cdot })\text{ .}
\end{equation*}%
Hence we have that%
\begin{eqnarray*}
	&&\sup_{h\in C^{3}}|\mathbb{E}\left[ h(T_{\ell ;t}(x_{1}),...,T_{\ell
		;t}(x_{d}))\right] -\mathbb{E}\left[ h(Z_{1},...,Z_{d})\right] | \\
	&=&\sup_{h\in C^{3}}|\mathbb{E}\left[ (h\circ \Psi_{\ell
		;d})(a_{\ell ,\cdot })\right] -\mathbb{E}\left[ (h\circ \Psi_{\ell ;d})(Z_{1},...,Z_{2\ell +1})\right] | \\
	&\leq &\sup_{h\in C^{3}}B_{3}(h\circ \Psi_{\ell ;d})\sqrt{8\sqrt{4\pi }\frac{1.539\log \ell }{\ell \nu _{t}}+O\left( \frac{1}{\ell \nu _{t}}\right) }
\end{eqnarray*}%
where%
\begin{eqnarray*}
	B_{3}(h\circ \Psi_{\ell ;d}) &=&\frac{\sqrt{2(2\ell +1)}}{4}
	M_{2}(h\circ \Psi_{\ell ;d})+\frac{2}{9}\sqrt{(2\ell +1)4\pi }
	M_{3}(h\circ \Psi_{\ell ;d}) \\
	&\leq &\frac{\sqrt{2(2\ell +1)}}{4}\sqrt{\frac{2\ell +1}{2\pi }}dM_{2}(h)+%
	\frac{2}{9}\sqrt{(2\ell +1)4\pi }\sqrt{\frac{2\ell +1}{2\pi }}dM_{3}(h) \\
	&=&d\frac{2\ell +1}{4\sqrt{\pi }}M_{2}(h)+\frac{2\sqrt{2}d}{9}%
	(2\ell +1)M_{3}(h)\text{ .}
\end{eqnarray*}%
Here we have used the simple fact that%
\begin{equation*}
\sup_{h:\mathbb{R}^{d}\rightarrow \mathbb{R}}M_{k}(h\circ \Psi_{\ell ;d}) \leq d\sqrt{\frac{2\ell +1}{2\pi }}\sup_{h}M_{k}(h)\text{ , }%
k\in \mathbb{N}\text{ .}
\end{equation*}%
Summing up, we have here a bound of order $\frac{d\sqrt{\ell \log \ell }}{%
	\sqrt{\nu _{t}}}$, to be compared with the bound of order $d^{2}\sqrt{\frac{%
		\log \ell }{\nu _{t}}}$ which was obtained in Theorem \ref{multifdd}.
Overall, we can claim that%
\begin{eqnarray*}
	&&\sup_{h\in \mathcal{C}^{3}\text{ , }|h|_{\mathcal{C}^{3}}<1}|\mathbb{E}%
	\left[ h(T_{\ell ;t}(x_{1}),...,T_{\ell ;t}(x_{d}))\right] -\mathbb{E}\left[
	h(Z_{1},...,Z_{d})\right] | \\
	&=&O\left( d\times (\sqrt{\ell }\wedge d)\times \sqrt{\frac{\log \ell }{\nu
			_{t}}}\right) \text{ .}
\end{eqnarray*}

\subsection{Functional Convergence}

In the previous subsections we presented a number of quantitative
convergence results for sequences of random vectors, such as vectors of
spherical harmonic coefficients or points evaluations over a subset of $d$
points. It is natural to ask whether we can also obtain results for the
sequence of eigenfunctions $\left\{ T_{\ell ,t}(.)\right\} $ considered as
random elements in functional spaces; the answer is affirmative,
thanks to some very recent results in this direction in \cite{BCD}. We shall consider in particular $L^{2}(\mathbb{S}^{2})$ and the Sobolev space $W_{\alpha,2}(\mathbb{S}^2)$, $\alpha>0$, to distinguish the probability metric in the two cases, we shall write $d_{3,L^{2}(\mathbb{S}^{2})}$ and $d_{3,W_{\alpha,2} (\mathbb{S}^2)}$, respectively.
Let us recall also that for a general function space $K$ we have that
$C_b^3(K)$ is the class of real-valued functions on $K$ that have bounded Fr\'echet derivatives up to order three. This space is equipped with the norm 
\[ ||h||_{C_b^3(K)}= \sup_{j=1,2,3} \sup_{x \in K} ||D^j h(x)||_{K^{\otimes j}} .\]

\subsubsection{Quantitative Central Limit Theorems in $L^{2}(\mathbb{S}^{2}).$}

We start by considering the space of $L^{2}(\mathbb{S}^{2})$. Our main result is the following.

\begin{theorem}
	Let $Z$ be a centred Gaussian process with the same covariance operator as $T_{\ell;t}$. We have that
	\begin{equation*}
	d_{3,L^{2}(\mathbb{S}^{2})}(T_{\ell ;t},Z)\leq \left( \frac{1}{4}+\sqrt{\pi }\right) \sqrt{\frac{%
			4\pi }{\nu _{t}}}
	\end{equation*}
\end{theorem}

\begin{proof} In view of Theorem \ref{FMT-BCD} (see \cite{BCD}, Theorem 3
	and Corollary 1), we need to compute the quantity
	\begin{equation*}
	\mathbb{E}[\left\Vert T_{\ell ;t}\right\Vert ^{4}_{L^{2}(\mathbb{S}^{2})}]-(\mathbb{E}[\left\Vert T_{\ell ;t}\right\Vert
	^{2}_{L^{2}(\mathbb{S}^{2})}])^{2}-2\left\Vert S_{\ell;t }\right\Vert _{HS(L^2 (\mathbb{S}^{2}))}^{2},
	\end{equation*}%
	where $S_{\ell ;t}$ is the covariance operator of $T_{\ell;t }$ and $||\cdot||_{HS}$ denotes the Hilbert-Schmidt norm (see the end of Appendix \ref{sec:generalresults}). First, we have
	that
	
	\begin{equation*}
	\mathbb{E}[||T_{\ell;t }||^{2}_{L^{2}(\mathbb{S}^{2})}]=\mathbb{E}\left[\int_{\mathbb{S}^{2}}|T_{\ell;t
	}(x)|^{2}\,dx\right]=\int_{\mathbb{S}^{2}}\sum_{m_{1}=-\ell}^{\ell}
	\sum_{m_{2}=-\ell}^{\ell}\mathbb{E}[\hat{a}_{\ell, m_{1}}(t)\ \hat{a}_{\ell
		,m_{2}}(t)]Y_{\ell m_{1}}(x)Y_{\ell m_{2}}(x)\,dx
	\end{equation*}%
	\begin{eqnarray*}
		&=&\int_{\mathbb{S}^{2}}\frac{4\pi }{2\ell +1}\sum_{m=-\ell }^{\ell }Y_{\ell
			m}(x)Y_{\ell m}(x)\,dx \\
		&=&\frac{4\pi }{2\ell +1}\sum_{m=-\ell }^{\ell }\int_{\mathbb{S}^{2}}Y_{\ell
			m}(x)Y_{\ell m}(x)dx=4\pi \text{ .}
	\end{eqnarray*}%
	It follows that $(\mathbb{E}[||T_{\ell ;t}||^{2}_{L^{2}(\mathbb{S}^{2})}])^{2}=(4\pi )^{2}.$ Now we
	compute $\mathbb{E}[||T_{\ell ;t}||^{4}_{L^{2}(\mathbb{S}^{2})}]$, which gives
	
	\begin{equation*}
	\mathbb{E}[||T_{\ell ;t}||^{4}_{L^{2}(\mathbb{S}^{2})}]=\mathbb{E}\left[||T_{\ell;t }||^{2}_{L^{2}(\mathbb{S}^{2})}||T_{\ell;t }||^{2}_{L^{2}(\mathbb{S}^{2})}\right]=%
	\mathbb{E}\left[\sum_{m_{1}=-\ell }^{\ell }|\hat{a}_{\ell,
		m_{1}}(t)|^{2}\sum_{m_{2}=-\ell }^{\ell }|\hat{a}_{\ell, m_{2}}(t)|^{2}\right]
	\end{equation*}%
	\begin{equation*}
	=\left( \frac{4\pi }{(2\ell +1)\nu _{t}}\right) ^{2}\mathbb{E}%
	\left[\sum_{m_{1}=-\ell }^{\ell }\sum_{k_{1}k_{2}}Y_{\ell m_{1}}(\xi
	_{k_{1}})Y_{\ell m_{1}}(\xi _{k_{2}})\sum_{m_{2}=-\ell }^{\ell
	}\sum_{k_{3}k_{4}}Y_{\ell m_{2}}(\xi _{k_{3}})Y_{\ell m_{2}}(\xi _{k_{4}})\right].
	\end{equation*}%
	Applying the addition formula we get
	\begin{eqnarray*}
		\mathbb{E}[||T_{\ell;t }||^{4}_{L^{2}(\mathbb{S}^{2})}]&=&\left( \frac{1}{\nu _{t}}\right) ^{2}\mathbb{E}%
		\left[\sum_{k_{1}=1}^{N_t(\mathbb{S}^2)} \sum_{k_{2}=1}^{N_t(\mathbb{S}^2)}P_{\ell }(\langle \xi _{k_{1}},\xi
		_{k_{2}}\rangle )\sum_{k_{3}=1}^{N_t(\mathbb{S}^2)} \sum_{k_{4}=1}^{N_t(\mathbb{S}^2)}P_{\ell }(\langle
		\xi _{k_{3}},\xi _{k_{4}}\rangle )\right]\\
		&=&\left( \frac{1}{\nu _{t}}\right) ^{2}\mathbb{E}\left[\sum_{k_{1}=1}^{N_t(\mathbb{S}^2)}P_{\ell
		}(\langle \xi _{k_{1}},\xi _{k_{1}}\rangle )^{2}\right]\\&&+\left( \frac{1}{\nu _{t}}%
		\right) ^{2}\mathbb{E}\left[\sum_{k_{1}=k_{2}\neq k_{3}=k_{4}}P_{\ell }(\langle
		\xi _{k_{1}},\xi _{k_{2}}\rangle )P_{\ell }(\langle \xi _{k_{3}},\xi
		_{k_{4}}\rangle )\right]\\
		&&	+\left( \frac{1}{\nu _{t}}\right) ^{2}\mathbb{E}\left[\sum_{k_{1}=k_{3}\neq
			k_{2}=k_{4}}P_{\ell }(\langle \xi _{k_{1}},\xi _{k_{2}}\rangle )P_{\ell
		}(\langle \xi _{k_{3}},\xi _{k_{4}}\rangle )\right]\\&&
		+\left( \frac{1}{\nu _{t}}\right) ^{2}\mathbb{E}\left[\sum_{k_{1}=k_{4}\neq
			k_{3}=k_{2}}P_{\ell }(\langle \xi _{k_{1}},\xi _{k_{2}}\rangle )P_{\ell
		}(\langle \xi _{k_{3}},\xi _{k_{4}}\rangle )\right]
	\end{eqnarray*}
	and since $P_{\ell }(0)=1$ for all $\ell $ we obtain%
	\begin{eqnarray*}
		\mathbb{E}[||T_{\ell ;t}||^{4}_{L^{2}(\mathbb{S}^{2})}]	&=&\left( \frac{1}{\nu _{t}}\right) ^{2}\mathbb{E}\left[\sum_{k_{1}=1}^{N_t(\mathbb{S}^2)}1\right]+%
		\left( \frac{1}{\nu _{t}}\right) ^{2}\mathbb{E}\left[\sum_{k_{1}=k_{2}\neq
			k_{3}=k_{4}}1\right] \\
		&&
		+2\left( \frac{1}{\nu _{t}}\right) ^{2}\mathbb{E}\left[\sum_{k_{1}=k_{3}\neq
			k_{2}=k_{4}}P_{\ell }(\langle \xi _{k_{1}},\xi _{k_{2}}\rangle )^{2}\right]
		\\
		&=&\frac{4\pi }{\nu _{t}}+(4\pi )^{2}\left( \frac{1}{\nu _{t}}\right)
		^{2}\nu _{t}^{2}
		+\left( \frac{1}{\nu _{t}}\right) ^{2}2\nu _{t}^{2}\int_{(\mathbb{S}%
			^{2})^{2}}P_{\ell }(\langle \xi _{k_{1}},\xi _{k_{2}}\rangle )^{2}\,d\xi
		_{k_{1}}\,d\xi _{k_{2}}\\
		&=&\frac{4\pi }{\nu _{t}}+(4\pi )^{2}+2(4\pi )\frac{4\pi }{2\ell +1}.
	\end{eqnarray*}%
	The covariance operator $S_{\ell ;t}$ is such that
	\begin{equation*}
	||S_{\ell;t }||_{HS(L^2(\mathbb
		S^2))}^{2}=\sum_{m=-\ell}^{\ell} \sum_{m^\prime=-\ell}^{\ell}
	\mathbb{E}[a_{\ell, m}(t)a_{\ell ,m^{\prime }}(t)]^{2}=\sum_{m=-\ell}^{\ell}
	\sum_{m^{\prime }=-\ell}^{\ell} \left( \delta _{m}^{m^{\prime }}\frac{4\pi }{%
		2\ell +1}\right) ^{2}=\frac{(4\pi )^{2}}{2\ell +1}\text{ ,}
	\end{equation*}%
	and then we finally obtain%
	\begin{eqnarray*}
		&&\mathbb{E}[\left\Vert T_{\ell ;t}\right\Vert ^{4}_{L^{2}(\mathbb{S}^{2})}]-(\mathbb{E}[\left\Vert T_{\ell ;t}\right\Vert_{L^{2}(\mathbb{S}^{2})}
		^{2}])^{2}-2\left\Vert S_{\ell;t }\right\Vert _{HS(L^2(\mathbb
			S^2))} \\
		&&=\frac{4\pi }{\nu _{t}}+(4\pi )^{2}+2\frac{(4\pi )^{2}}{2\ell +1}-(4\pi
		)^{2}-2\frac{(4\pi )^{2}}{2\ell +1}=\frac{4\pi }{\nu _{t}}\text{ .}
	\end{eqnarray*}%
	Exploiting Theorem \ref{FMT-BCD} (see also \cite{BCD}) we get the thesis of
	the theorem.
\end{proof}

As mentioned above, it may come at first sight as a surprise that the rate of
convergence in this functional setting (i.e., $1/\sqrt{\nu _{t}})$ does not
depend on the index $\ell $ and it is indeed faster than in the
finite-dimensional case. The apparent paradox is solved noting that the topology here is too coarse to imply convergence of the finite-dimensional distributions. In the next subsection, we investigate convergence in functional spaces with a finer topological structure.

\subsubsection{Quantitative Central Limit Theorems in $W_{\alpha,2}(\mathbb{S}^{2}).$} \label{2.3.2}

Now we consider the random eigenfunctions taking values in Sobolev
spaces $W_{\alpha,2}(\mathbb{S}^{2})$, $\alpha >0$, on the sphere, i.e., the spaces of functions $f \in L(\mathbb{S}^2)$,  $f=\sum_{\ell,m}a_{\ell,m}Y_{\ell,m}$, with finite norm
\begin{equation*}
\left\| f \right\|^2_{W_{\alpha,2} (\mathbb{S}^2)}=\sum_{\ell\geq 0} \sum_{m=-\ell}^{\ell} \left(1+\sqrt{\ell(\ell+1)}\right)^{2\alpha}\left \vert a_{\ell,m}\right \vert^2.
\end{equation*}
Our main result here is the following.
\begin{theorem}\label{Sobolev}
	Let $Z$ be a centred Gaussian process with the same covariance operator as $T_{\ell;t}$. We have that
	\begin{equation*}
	d_{3,W_{\alpha,2}}(T_{\ell ;t},Z)\leq
	\frac{\sqrt{\pi} \left(1+\sqrt{\ell(\ell+1)}\right)^{2\alpha}}{2\sqrt{\nu
			_{t}}} + \frac{2\pi \left(1+\sqrt{\ell(\ell+1)}\right)^{3\alpha}}{\sqrt{\nu _{t}}}.
	\end{equation*}
\end{theorem}
\begin{proof}
	First note that 
	\begin{equation*}
	\mathbb{E}\left[\left\| T_{\ell;t}\right\|_{W_{\alpha,2}(\mathbb{S}^2)}^{4}\right] =
	\left(1+\sqrt{\ell(\ell+1)}\right)^{4\alpha}\mathbb{E}\left[||T_{\ell;t }||^{4}_{L^2(\mathbb{S}^2)}\right]
	\end{equation*}
	and 
	\begin{equation}\label{W2}
	\mathbb{E}\left[\left\| T_{\ell;t}\right\|_{W_{\alpha,2}(\mathbb{S}^2)}^{2}\right] =
	\left(1+\sqrt{\ell(\ell+1)}\right)^{2\alpha}\mathbb{E}\left[||T_{\ell;t }||^{2}_{L^2(\mathbb{S}^2)}\right].
	\end{equation}
	Indeed, we have that
	\begin{eqnarray*}
		\mathbb{E}\left[\left\| T_{\ell;t}\right\|_{W_{\alpha,2}(\mathbb{S}^2)}^{4}\right] &=& 	\mathbb{E}\left[\left\| T_{\ell;t}\right\|_{W_{\alpha,2}(\mathbb{S}^2)}^{2}\left\| T_{\ell;t}\right\|_{W_{\alpha,2}(\mathbb{S}^2)}^{2}\right]\\
		&=& 	\mathbb{E}\left[\sum_{m=-\ell}^{\ell}\sum_{m'=-\ell}^{\ell}  \left(1+\sqrt{\ell(\ell+1)}\right)^{4\alpha} \left\vert \widehat{a}_{\ell,m}\right\vert^2\left\vert \widehat{a}_{\ell,m^\prime}\right\vert^2\right]\\
		&=& \left(1+\sqrt{\ell(\ell+1)}\right)^{4\alpha}\mathbb{E}\left[\sum_{m=-\ell}^{\ell}\sum_{m'=-\ell}^{\ell}   \left\vert \widehat{a}_{\ell,m}\right\vert^2\left\vert \widehat{a}_{\ell,m^\prime}\right\vert^2\right]\\
		&=&\left(1+\sqrt{\ell(\ell+1)}\right)^{4\alpha}\mathbb{E}\left[||T_{\ell;t }||^{4}_{L^{2}(\mathbb{S}^{2})}\right],
	\end{eqnarray*}
	where in the last equation we used Parseval's identity. Similarly (\ref{W2}) holds.
	In view of the computations of the previous section and this remark, we conclude that
	\[ 	\mathbb{E}\left[\left\| T_{\ell;t}\right\|_{W_{\alpha,2}(\mathbb{S}^2)}^{4}\right] =\left(1+\sqrt{\ell(\ell+1)}\right)^{4\alpha} \left( \frac{4\pi }{\nu _{t}}+(4\pi )^{2}+2(4\pi )\frac{4\pi }{2\ell +1}\right)\] 
	and 
	\[ 	\mathbb{E}\left[\left\| T_{\ell;t}\right\|_{W_{\alpha,2}(\mathbb{S}^2)}^{2}\right] =4\pi\left(1+\sqrt{\ell(\ell+1)}\right)^{2\alpha}. \]
	Furthermore, letting $\left\{ e_i \colon i \geq 1 \right\}$ be an
	orthonormal basis of $W_{\alpha,2}(\mathbb{S}^{2})$, we can compute
	$\left\lVert S_{\ell;t} \right\rVert^2_{HS(W_{\alpha,2})}$ as 
	\begin{align*}
	\left\lVert S_{\ell;t} \right\rVert^2_{HS(W_{\alpha,2})} &= \sum_{i\geq
		1}^{}\left\lVert \mathbb{E}\left( \left\langle T_{\ell;t},e_i
	\right\rangle_{W_{\alpha,2}(\mathbb{S}^{2})}T_{\ell;t}
	\right)
	\right\rVert^2_{W_{\alpha,2}(\mathbb{S}^2)}\\
	&=\sum_{i \geq 1}^{}\left\lVert \sum_{m=-\ell}^{\ell}\mathbb{E}\left(\widehat{a}_{\ell,m}(t)^2
	\right)\left\langle Y_{\ell m},e_i
	\right\rangle_{W_{\alpha,2}(\mathbb{S}^{2})}Y_{\ell
		m}\right\rVert^2_{W_{\alpha,2}(\mathbb{S}^2)}\\
	&= \sum_{m=-\ell}^{\ell}\left( 1+ \sqrt{\ell(\ell+1)}
	\right)^{2\alpha}\left( \frac{4\pi}{2\ell+1} \right)^2 \sum_{i
		\geq 1}^{}\left\langle Y_{\ell m},e_i
	\right\rangle^2_{W_{\alpha,2}(\mathbb{S}^2)}\\
	&= \sum_{m=-\ell}^{\ell}\left( 1+ \sqrt{\ell(\ell+1)}
	\right)^{4\alpha}\left( \frac{4\pi}{2\ell+1} \right)^2\\
	&= \left( 1+ \sqrt{\ell(\ell+1)}
	\right)^{4\alpha} \frac{(4\pi)^2}{2\ell+1}.
	\end{align*}
	We now have all the necessary elements to apply Theorem \ref{FMT-BCD}
	as in the previous subsection,
	from which the result follows after elementary algebraic manipulations.
\end{proof}
As a final result we want to show that, for $\alpha>\frac{3}{2}$, a
quantitative Central Limit Theorem in Sobolev space does indeed imply
the quantitative Central Limit Theorem for the marginal distribution
at every given location on the sphere. We start by noting that
\begin{eqnarray*}
	\left\Vert f\right\Vert _{L^{\infty }(\mathbb{S}^2)} &=&\sup_{x}|\sum_{\ell
	}\sum_{m}a_{\ell m}(f)Y_{\ell m}(x)| \\
	&\leq &\sum_{\ell }\sum_{m}|a_{\ell m}(f)|\sup_{x}|Y_{\ell m}(x)| \\
	&\leq &\sum_{\ell }\sum_{m}|a_{\ell m}(f)|\sqrt{\frac{2\ell +1}{2\pi}},
\end{eqnarray*}%
whence%
\begin{eqnarray*}
	\left\Vert f\right\Vert _{L^{\infty }(\mathbb{S}^2)}^{2} &\leq &\frac{1}{2\pi }\left\{
	\sum_{\ell }\sum_{m}|a_{\ell m}(f)|\sqrt{2\ell +1}\right\} ^{2}. 
\end{eqnarray*}
Multiplying and dividing by
$(1+\sqrt{\ell(\ell+1)})^\alpha\sqrt{2\ell+1} $ and then applying the
Cauchy-Schwarz inequality twice, we get
\begin{eqnarray*}
	\left\Vert f\right\Vert _{L^{\infty }(\mathbb{S}^2)}^{2}
	&\leq &  \frac{1}{2\pi}
	\sum_{\ell } (2\ell+1) \sum_{m}|a_{\ell m}(f)|^2\frac{(1+\sqrt{\ell(\ell+1)})^{2\alpha} }{(2\ell+1)}	\sum_{\ell }  \frac{(2\ell +1)^2} {(1+\sqrt{\ell(\ell+1)})^{2\alpha}} \\
	&=&  \frac{1}{2\pi} ||f||_{W_{\alpha,2}(\mathbb{S}^2)}^2	\sum_{\ell }  \frac{(2\ell +1)^2} {(1+\sqrt{\ell(\ell+1)})^{2\alpha}} \\
	&\leq&  \frac{2}{\pi} ||f||_{W_{\alpha,2}(\mathbb{S}^2)}^2	\zeta(2\alpha-2),
\end{eqnarray*}
where as usual
\[
\zeta (2\alpha -2)=\sum_{\ell =1}^{\infty }\frac{1}{\ell ^{2\alpha -2}}<\infty
\]%
as $\alpha>\frac{3}{2}$. Hence, we have that%
\[
\left\Vert f\right\Vert _{L^{\infty }(\mathbb{S}^2)}^{2}<\frac{2}{\pi} \zeta(2\alpha-2)\times \left\Vert
f\right\Vert _{W_{\alpha,2}(\mathbb{S}^2)}^{2}.
\]%
Because of this inequality, the topology
induced by the norm $\left\Vert .\right\Vert _{W_{\alpha,2}(\mathbb{S}^2)}$ is
finer than the topology generated by $\left\Vert .\right\Vert _{L^{\infty
	}(\mathbb{S}^2)}$; hence a function continuous with respect to the latter is certainly continuous with respect to the former as well. Therefore%
\begin{equation*}
\sup_{h\text{ continuous w.r.t. }\left\Vert .\right\Vert _{L^{\infty
		}(\mathbb{S}^2)}}|\mathbb{E}h(X)-\mathbb{E}h(Y)| \leq \sup_{h\text{ continuous w.r.t. }\left\Vert \cdot\right\Vert _{W_{\alpha,2}(\mathbb{S}^2)}}|\mathbb{E}h(X)-\mathbb{E}h(Y)|\text{ .}
\end{equation*}%
Now we show that $d_{3,W_{\alpha,2}}(X_{\ell },Z_{\ell })\rightarrow 0$ implies%
\[
\mathbb{E}g(X_{\ell }(x))\rightarrow \mathbb{E}g(Z_{\ell }(x))\text{ for all
}g\in C_{b}^{3}(\mathbb{R)}\text{ ,}
\]
for fixed $x \in \mathbb{S}^2$, which in turn implies $X_{\ell }(x)\rightarrow _{d}N(0,1)$
because $d_{3}$ metrizes convergence in distribution, in particular on $%
\mathbb{R}.$ Actually we show the following, slighlty stronger result.

\begin{corollary}\label{final}
	For $\alpha>\frac{3}{2}$, we have that
	\[
	d_{3}(X_{\ell }(x),Z_{\ell }(x))=\sup_{g\in C_{b}^{3}(\mathbb{R)}}\left\vert
	\mathbb{E}g(X_{\ell }(x))-\mathbb{E}g(Z_{\ell }(x))\right\vert \leq
	C(\alpha)d_{3,W_{\alpha,2}}(X_{\ell },Z_{\ell })\text{ ,}
	\]%
	where the term $C(\alpha)$ does not depend on $\ell$.
\end{corollary}
\begin{proof}
	We can write%
	\[
	\mathbb{E}g(X_{\ell }(x))=\mathbb{E}(g\circ \pi _{x}(X_{\ell }(.)))\text{
		where }g\circ \pi _{x}\in C_{b}\text{ ,}
	\]%
	where the evaluation map $\pi _{x}:\pi _{x}(X_{\ell })=X_{\ell }(x)$ is continuous with respect to the
	Sobolev norm (because it is continuous with respect to the sup norm).
	
	Note that the Gateaux derivatives of the evaluation functionals are given by%
	\[
	\frac{|\pi _{x}(X_{\ell }+tH)-\pi _{x}(X_{\ell })|}{t}=H(x)=\pi _{x}(H)\text{
		, } \forall \text{ } H \in W_{\alpha,2}  \text{
		, } 
	\]%
	so that the Fr\'echet derivative ($(D\pi _{x})(X_{\ell }))H=H(x),$ that is ($%
	(D\pi _{x})(X_{\ell }))=\pi _{x}.$ Note also that the (dual) norm of $\pi _{x}$ is
	bounded, indeed by its definition we have that%
	\begin{eqnarray*}
		\left\Vert \pi _{x}(.)\right\Vert _{W_{\alpha,2}^{\ast }} &:&=\sup_{h:\left\Vert
			h\right\Vert _{W_{\alpha,2}}=1}|\pi _{x}(h)|=\sup_{h}\frac{|h(x)|}{\left\Vert
			h\right\Vert _{W_{\alpha,2}(\mathbb{S}^2)}} \\
		&\leq &\sup_{h}\frac{\left\Vert h\right\Vert _{L^{\infty }(\mathbb{S}^2)}}{%
			\left\Vert h\right\Vert _{W_{\alpha,2}(\mathbb{S}^2)}}\leq \frac{2}{\pi}\zeta (2\alpha-2)\text{ .}
	\end{eqnarray*}%
	Similar results are obtained if we take the second or third order
	Fr\'echet derivatives, with the same bound. Therefore we have that%
	\[
	d_{3}(X_{\ell }(x),Z_{\ell }(x))=\sup_{g\in C_{b}^{3}(\mathbb{R)}}\left\vert
	\mathbb{E}g(X_{\ell }(x))-\mathbb{E}g(Z_{\ell }(x))\right\vert \leq
	C(\alpha)d_{3,W_{2,\alpha}}(X_{\ell },Z_{\ell }),
	\]%
	which proves the claim with $C(\alpha):=\frac{2}{\pi}\zeta (2\alpha-2)$.
\end{proof}

%
%
%
%
%
%
%
%
%

\section{Appendix}\label{appendix}

In this Appendix, we collect for convenience a number of background results
on Fourth Moment Theorems in a Poisson environment and on integrals of
spherical harmonics. We start introducing some notation and definitions.

\subsection{Wiener chaos in a Poisson environment}

\label{sec:generalresults} We now present, in a form properly adapted to our
goals, some upper bounds related to random variables living in the first
Wiener chaos of a Poisson random measure. The first two bounds have been
proved in \cite{DVZ} and provide a Fourth Moment Theorem on the Poisson
space for the univariate and the multivariate case respectively. The third
bound appears in \cite{BCD} and concerns a quantitative and functional
Central Limit Theorem for convergence to a Gaussian process. \newline
We start by recalling some basic concepts on Poisson random measures and
Wiener chaos. Assuming that we are working on a suitable probability space, $%
\left( \Omega ,\mathcal{F},P\right) $, the following definition is standard:

\begin{definition}[Poisson random measure]
	Let $\left( \Theta ,\mathcal{A},\rho \right) $ be a $\sigma $-finite measure
	space, such that $\rho $ has no atoms. A Poisson random measure on $\Theta $
	with intensity measure $\rho $ is a collection of random variables $\left\{
	N(A):A\in \mathcal{A}\right\} $, taking values in the space $\mathcal{Z}%
	_{+}\cup \left\{ \infty \right\} $, characterized by the following two
	properties:
	
	\begin{enumerate}
		\item for every $A \in \mathcal{A}$, $N (A)$ has Poisson distribution with
		intensity $\rho(A)$;
		
		\item for $A_1,\ldots , A_n \in \mathcal{A}$ pairwise disjoint, $N (A_1),
		\ldots , N (A_n)$ are independent.
	\end{enumerate}
	
	The centred Poisson random measure $\hat{N}$ is defined by $\hat{N} := N -\rho$.
\end{definition}

From now on, for the sake of brevity, we will make use of the shorthand
notation $L^{p}\left( \rho \right) $ to denote the Lebesgue space $%
L^{p}\left( \Theta ,\mathcal{A},\rho \right) $, while, for $p\geq 2$, we
will denote with $L_{s}^{p}\left( \rho \right) \subset L^{p}\left( \rho
\right) $ the subspace of symmetric functions.

\begin{definition}[Wiener--It\^o integrals and first Wiener chaos]
	For every deterministic function $h \in L_2^2\left(\rho\right)$, the
	Wiener--It\^o integral of $h$ with respect to $\hat{N}$ is given by
	\begin{equation*}
	I_1\left(h\right) = \int_{\Theta} h\left(z\right)\hat{N}\left(dz\right).
	\end{equation*}
	The Hilbert space composed of the random variables of the form $%
	I_1\left(h\right)$, where $h \in L^2_s \left(\rho\right)$, and labeled by $%
	W_1$, is called the \emph{first Wiener chaos} associated with the Poisson
	measure $N$.
\end{definition}

In this paper, we choose $\Theta= \mathbb{R}_+ \times \mathbb{S}^2$, while $%
\mathcal{A}$ is the class of Borel subsets of $\Theta$, labeled by $%
\mathcal{B}(\Theta)$. We denote with $N_t$ a Poisson random measure on $%
\Theta$, whose intensity is given by the product measure $\rho=\lambda
\times \mu$. The first term, which can be read as the time component, is
given by $\lambda \left(ds \right) = \nu \times \ell(ds)$, where $\nu>0$ is
a fixed parameter, while $\ell$ is the Lebesgue measure, so that for any $t
\in \mathbb{R}$, $\lambda\left(\left[0, t\right]\right) := \nu \times t =
\nu_t$. Regarding the spherical component, $\mu$ is assumed to be a
probability measure on $\mathbb{S}^2$, associated to a density $f$ so that $%
\mu(dx) = f(x)dx$. Given this setting, we will denote by $N_t$, $t > 0$, the
Poisson measure on $\left(\mathbb{S}^2,\mathcal{B}\left(\mathbb{S}%
^2\right)\right)$ defined by $N_t\left(B\right) := N\left(\left[0, t\right]
\times B\right)$, $B \in \mathcal{B} \left(\mathbb{S}^2\right)$ with
intensity $\mu_t=\nu_t \times \mu$.

\begin{theorem}
	\label{FMT-DP}[Quantitative Fourth Moment Theorem (unidimensional case),
	\cite[Theorem 2.1 and Corollary 1.3]{DVZ} and \cite[Theorem 1.3]{DP}] For $%
	\ell \in \mathbb{N}$, let $F\in W_{1}$, while $Z\sim \mathcal{N}(0,1)$
	denote a standard normal random variable. Moreover, assume that $Var(F)=1$
	and $\mathbb{E}[F^{4}]<\infty $. Then it holds that:
	\begin{equation*}
	d_{W}(F,Z)\leq c_{1}\sqrt{\mathbb{E}[F^{4}]-3},
	\end{equation*}%
	where
	\begin{equation*}
	c_{1}:=\frac{1}{\sqrt{2\pi }}+\frac{2}{3}.
	\end{equation*}%
	Moreover, it holds that
	\begin{equation*}
	d_{Kol}(F,N)\leq \left( 11+(\mathbb{E}[F^{4}])^{1/2}+(\mathbb{E}%
	[F^{4}])^{1/4}\right) \sqrt{\mathbb{E}[F^{4}]-3}.
	\end{equation*}
\end{theorem}

Before stating the next result, we need some additional notation. For any $%
\ell \in \mathbb{N}$, fixed an integer $d\geq 2$, we consider the centred
random vector $F=(F_{1},\dots ,F_{d})^{T}$ where $F_{j}\in W_{1}$, for $%
1\leq j\leq d$. For $j=1,\ldots ,d$. We denote by $\Gamma _{d}$ the
covariance matrix of $F$, i.e. $\Gamma _{d;ij}=\mathbb{E}[F_{i}F_{j}]$ for $%
i,j=1,\dots ,d$.\newline
For a $k-$multilinear form $\psi :(\mathbb{R}^{d})^{k}\rightarrow \mathbb{R}$%
, $k\in \mathbb{N}$, we define the \textit{operator norm}
\begin{equation*}
||\psi ||_{op}:=\sup \{|\psi (u_{1},\dots ,u_{k})|:u_{j}\in \mathbb{R}%
^{d},||u_{j}||_{2}=1,j=1,\dots ,k\}.
\end{equation*}%
Furthermore, note that $D^{k-1}g(x)$, the $(k-1)$--derivative of $g$ at the
point $x$, can be read as a multilinear form. In this setting, we can define
the generalization of the minimum Lipschitz constant for any derivative of
order $k-1$ as follows: fixed $k\geq 1$ and chosen $g\in C^{k-1}\left(
\mathbb{R}^{d}\right) $, take%
\begin{equation*}
M_{k}(g):=\sup_{x\neq y}\frac{||D^{k-1}g(x)-D^{k-1}g(y)||_{op}}{%
	||x-y||_{\mathbb{R}^{d} }},
\end{equation*}%
see again \cite{DVZ}.\newline

\begin{theorem}
	\label{FMT-DVZ}[Quantitative Fourth Moment Theorem (multidimensional case),
	\cite[Theorem 1.7, Corollary 1.8 and Remark 4.3]{DVZ}] Under the above
	notation, let $Z_d$ be a centred Gaussian random vector of dimension $d$
	with covariance matrix $\Gamma_d$. Then, for every $g\in C^{3}(\mathbb{R}%
	^{d})$, we have that
	\begin{equation*}
	|\mathbb{E}[g(F)]-\mathbb{E}[g(Z_d)]|\leq B_{3}(g;d)\sum_{i=1}^{d}\sqrt{%
		\mathbb{E}[F_{i}^{4}]-3\mathbb{E}[F_{i}^{2}]^{2}}
	\end{equation*}%
	where%
	\begin{equation}  \label{B3}
	B_{3}(g;d)=A_{2}(g;d)+\frac{2\sqrt{d \text{Tr}(\Gamma_d)}}{9} M_{3}(g)\text{ , }%
	A_{2}(g;d)=\frac{\sqrt{2d}}{4}M_{2}(g)\text{ }.
	\end{equation}
	If in addition $\Gamma_d$ is positive definite, then for every $g\in C^{2}(
	\mathbb{R}^{d})$, it holds that
	\begin{equation*}
	|\mathbb{E}[g(F)]-\mathbb{E}[g(Z)]| \leq B_2(g;d) \sum_{i=1}^{d} \sqrt{%
		\mathbb{\ E}[F_i^4]-3\mathbb{E}[F_i^2]^2}
	\end{equation*}
	with
	\begin{equation*}
	B_{2}(g;d)=A_{1}(g;d)+\frac{\sqrt{2\pi}||\Gamma_d^{-\frac{1}{2}}||_{op}
		\text{Tr}(\Gamma_d )}{6} M_{2}(g)\text{ , } A_{1}(g;d)=\frac{||\Gamma_d^{-\frac{1}{2%
		}}||_{op}}{\sqrt{\pi}}M_{1}(g)\text{ }.
	\end{equation*}
\end{theorem}

Let $K$ be a separable Hilbert space and $X$ a $K-$valued random variable in $L^2(\rho)$. We recall that if $X \in L^2(\rho)$, with $\mathbb{E}\left[\left \vert \left \vert X \right \vert\right\vert_K^2 \right]<\infty$, the covariance operator $S: K \to K$ of $X$ is defined by
$$Su=\mathbb{E}[\langle X,u \rangle_{K} X].$$
$S$ is a positive, self-adjoint trace-class operator that verifies the identity $$\text{Tr} S= \mathbb{E}[||X||_{K }^2].$$
We consider the Banach space of all trace-class operators on $K$, equipped with norm $\text{Tr}|A|$, where $|A|=\sqrt{A^*A}$ and $A^*$ denotes the adjoint of $A$. The subspace of Hilbert-Schmidt operators on $K$ is denoted by HS($K$), associated to the norm $||A||_{\text{HS}(K)} = \sqrt{\text{Tr}(AA^*)}$, $A \in \text{HS}(K)$.\\

Now, assume that $X$ is a $K-$valued random variable which belongs to the first
Wiener chaos with finite fourth moment, i.e. $\mathbb{E}[||X||_{K}^4]<\infty$,
and with covariance operator $S$.
We denote by $Z$ a Gaussian process taking values in the same separable
Hilbert space as $X$ and having the same covariance operator $S$. The following result holds.

\begin{theorem}
	\label{FMT-BCD}[Functional Quantitative Fourth Moment Theorem, \cite[Theorem
	3 and Corollary 1]{BCD}] Under the above notation and assumptions, it holds
	that
	\begin{equation*}
	d_{3}(X,Z)\leq \left( \frac{1}{4}+\frac{1}{2}\sqrt{ \mathbb{E}[||X||_{K}^2]}%
	\right) \sqrt{\mathbb{E}[||X||_{K}^{4}]-\mathbb{E}%
		[||X||_{K}^{2}]^{2}-2||S||_{HS(K)}^{2}}.
	\end{equation*}
\end{theorem}

\subsection{Wigner's and Clebsch-Gordan coefficients}

\label{C-G coefficients} In this section we review briefly some background
facts and notation about Wigner's 3j and Clebsch-Gordan coefficients; see
\cite{Varshalovich} and \cite{MP} for a much more detailed discussion, in
particular concerning the relationships with the quantum theory of angular
momentum and group representation properties of $SO(3)$.

We start recalling the following analytic expression for the Wigner's 3j
coefficients (valid for $m_{1}+m_{2}+m_{3}=0$, see \cite{Varshalovich}, eq.
(8.2.1.5))%
\begin{equation*}
\begin{pmatrix}
\ell _{1} & \ell _{2} & \ell _{3} \\
m_{1} & m_{2} & m_{3}%
\end{pmatrix}%
:=(-1)^{\ell _{1}+m_{1}}\sqrt{2\ell _{3}+1}\left[ \frac{(\ell _{1}+\ell
	_{2}-\ell _{3})!(\ell _{1}-\ell _{2}+\ell _{3})!(\ell _{1}-\ell _{2}+\ell
	_{3})!}{(\ell _{1}+\ell _{2}+\ell _{3}+1)!}\right] ^{1/2}
\end{equation*}%
\begin{equation*}
\times \left[ \frac{(\ell _{3}+m_{3})!(\ell _{3}-m_{3})!}{(\ell
	_{1}+m_{1})!(\ell _{1}-m_{1})!(\ell _{2}+m_{2})!(\ell _{2}-m_{2})!}\right]
^{1/2}
\end{equation*}%
\begin{equation*}
\times \sum_{z}\frac{(-1)^{z}(\ell _{2}+\ell _{3}+m_{1}-z)!(\ell
	_{1}-m_{1}+z)!}{z!(\ell _{2}+\ell _{3}-\ell _{1}-z)!(\ell
	_{3}+m_{3}-z)!(\ell _{1}-\ell _{2}-m_{3}+z)!},
\end{equation*}%
where the summation runs over all $z^{\prime }$s such that the factorials
are non-negative. This expression becomes simpler when $m_{1}=m_{2}=m_{3}=0,$
where we have%
\begin{equation*}
\begin{pmatrix}
\ell _{1} & \ell _{2} & \ell _{3} \\
0 & 0 & 0%
\end{pmatrix}%
=
\end{equation*}%
\begin{equation}
\begin{cases}
0, &  \\
\mbox{ }\mbox{ }\mbox{ }\mbox{ }\mbox{ }\mbox{ }\mbox{ }\mbox{ }\mbox{ }%
\mbox{ }\mbox{ }\mbox{ }\mbox{ }\mbox{ }\mbox{ }\mbox{ }\mbox{ }\mbox{ }%
\mbox{ }\mbox{ }\mbox{ }\mbox{ }\mbox{ }\mbox{ }\mbox{ }\mbox{ }\mbox{ }%
\mbox{ }\mbox{ }\mbox{ }\mbox{ }\mbox{ for }\ell _{1}+\ell _{2}+\ell _{3}%
\mbox{ odd} &  \\
\frac{(-1)^{\frac{\ell _{1}+\ell _{2}-\ell _{3}}{2}}[(\ell _{1}+\ell
	_{2}+\ell _{3})/2]!}{[(\ell _{1}+\ell _{2}-\ell _{3})/2]![(\ell _{1}-\ell
	_{2}+\ell _{3})/2]![(-\ell _{1}+\ell _{2}+\ell _{3})/2]!}\big\{\frac{(\ell
	_{1}+\ell _{2}-\ell _{3})!(\ell _{1}-\ell _{2}+\ell _{3})!(-\ell _{1}+\ell
	_{2}+\ell _{3})!}{(\ell _{1}+\ell _{2}+\ell _{3}+1)!}\big\}^{1/2}, &  \\
\mbox{ }\mbox{ }\mbox{ }\mbox{ }\mbox{ }\mbox{ }\mbox{ }\mbox{ }\mbox{ }%
\mbox{ }\mbox{ }\mbox{ }\mbox{ }\mbox{ }\mbox{ }\mbox{ }\mbox{ }\mbox{ }%
\mbox{ }\mbox{ }\mbox{ }\mbox{ }\mbox{ }\mbox{ }\mbox{ }\mbox{ }\mbox{ }%
\mbox{ }\mbox{ }\mbox{ }\mbox{ }\mbox{ for }\ell _{1}+\ell _{2}+\ell _{3}%
\mbox{ even.} &
\end{cases}
\label{expression}
\end{equation}
On the other hand the so-called Clebsch-Gordan coefficients, denoted by $%
\{C_{\ell _{1},m_{1};\ell _{2},m_{2}}^{\ell _{3},m_{3}}\}$, are defined by
the identities (see \cite{Varshalovich}, Chapter 8)%
\begin{equation}
\begin{pmatrix}
\ell _{1} & \ell _{2} & \ell _{3} \\
m_{1} & m_{2} & m_{3}%
\end{pmatrix}%
=(-1)^{\ell _{3}+m_{3}}\dfrac{1}{\sqrt{2\ell _{3}+1}}C_{\ell
	_{1},-m_{1};\ell _{2},-m_{2}}^{\ell _{3},m_{3}}  \label{3.67 libro Domenico}
\end{equation}%
\begin{equation}
C_{\ell _{1},m_{1};\ell _{2},m_{2}}^{\ell _{3},m_{3}}=(-1)^{\ell _{1}-\ell
	_{2}+m_{3}}\sqrt{2\ell _{3}+1}%
\begin{pmatrix}
\ell _{1} & \ell _{2} & \ell _{3} \\
m_{1} & m_{2} & -m_{3}%
\end{pmatrix}%
.  \label{3.68 libro Domenico}
\end{equation}
The following orthonormality properties hold and are exploited in this paper:%
\begin{equation}
\sum_{m_{1}m_{2}}C_{\ell _{1},m_{1};\ell _{2},m_{2}}^{\ell
	_{3},m_{3}}C_{\ell _{1},m_{1};\ell _{2},m_{2}}^{\ell _{3}^{\prime
	},m_{3}^{\prime }}=\delta _{\ell _{3}}^{\ell _{3}^{\prime }}\delta
_{m_{3}}^{m_{3}^{\prime }},  \label{unitary 1}
\end{equation}%
\begin{equation}
\sum_{\ell m}C_{\ell _{1},m_{1};\ell _{2},m_{2}}^{\ell m}C_{\ell
	_{1},m_{1}^{\prime };\ell _{2},m_{2}^{\prime }}^{\ell ,m}=\delta
_{m_{1}}^{m_{1}^{\prime }}\delta _{m_{2}}^{m_{2}^{\prime }}.
\label{unitary 2}
\end{equation}%
Note that the Clebsch-Gordan coefficients vanish unless the so-called
triangular conditions
\begin{equation*}
|\ell _{1}-\ell _{2}|\leq \ell _{3}\leq \ell _{1}+\ell _{2},\quad
\mbox{ and
	the equation }\quad m_{1}+m_{2}=m_{3}
\end{equation*}%
are satisfied (see \cite{Varshalovich}, Chapter 8). For some special values
of the arguments, namely if $\ell _{3}=0\mbox{ or }\ell _{2}=0$, one has
more explicit forms of these coefficients:
\begin{equation}
C_{\ell _{1},m_{1};\ell _{2},m_{2}}^{0,0}=(-1)^{\ell _{1}-m_{1}}\dfrac{%
	\delta _{\ell _{1}}^{\ell _{2}}\delta _{m_{1}}^{-m_{2}}}{\sqrt{2\ell _{1}+1}}%
.  \label{1 pag 248}
\end{equation}
From Remark 8.9 in \cite{MP} we also have the inequality
\begin{equation}
\{C_{\ell ,0;\ell ,0}^{L,0}\}^{2}\leq \frac{1}{(2L+1)}\text{ .}
\label{rem8.9MP}
\end{equation}
Now we recall the general formula (\cite{Varshalovich}, eqs. 5.6.2.12-13, or
\cite{MP} eqs 3.64 and 6.46) for the evaluation of multiple integrals of
spherical harmonics, the so-called \textit{Gaunt integrals}, given by%
\begin{eqnarray*}
	&&\int_{\mathbb{S}^{2}}Y_{\ell _{1}m_{1}}(x).....Y_{\ell _{n}m_{n}}(x)\,dx\\
	&&=\sqrt{\frac{4\pi }{2\ell _{n}+1}}\sum_{L_{1}...L_{n-3}}%
	\sum_{M_{1}...M_{n-3}}\left[ C_{\ell _{1},m_{1};\ell
		_{2},m_{2}}^{L_{1},M_{1}}C_{L_{1},M_{1};\ell
		_{3},m_{3}}^{L_{2},M_{2}}...C_{L_{n-3},M_{n-3};\ell _{n-1},m_{n-1}}^{\ell
		_{n},-m_{n}}\right. \\
	&&\times \left. \sqrt{\frac{\prod_{i=1}^{n-1}(2\ell _{i}+1)}{(4\pi )^{n-1}}}%
	\{C_{\ell _{1},0;\ell _{2},0}^{L_{1},0}C_{L_{1},0;\ell
		_{3},0}^{L_{2},0}...C_{L_{n-3},0;\ell _{n-1},0}^{\ell _{n},0}\}\right] .
\end{eqnarray*}
The most important case dealt with in this paper is given by%
\begin{eqnarray*}
	&&\int_{\mathbb{S}^{2}}Y_{\ell m_{1}}(x)Y_{\ell m_{2}}(x)Y_{\ell
		m_{3}}(x)Y_{\ell m_{4}}(x)\,dx\\
	&&= \frac{(2\ell+1)}{\sqrt{4\pi}} \sum_{L} (-1)^{L-M}\{C_{\ell ,0 ;\ell
		,0}^{L,0}\}^2 \frac{C_{\ell ,m_1; \ell, m_2}^{L,M} C_{\ell, m_3;\ell,
			m_4}^{L,-M}}{2L+1}.
\end{eqnarray*}

Similarly, as shown in \cite[Eq. (30)]{M e W 2011}, the following identity
hold%
\begin{eqnarray*}
	\int_{0}^{1}P_{l}(t)^{4}\,dt&=&\frac{1}{2\ell +1}\sum_{L=0}^{2\ell
	}\{C_{l,0;l,0}^{L,0}C_{L,0;l,0}^{L,0}\}^{2}\\
	&=&\sum_{L=0}^{2l}(2L+1)%
	\begin{pmatrix}
		l & l & L \\
		0 & 0 & 0%
	\end{pmatrix}%
	^{4}
\end{eqnarray*}
compare \cite[Eq. (8.9.4.20)]{Varshalovich}.


\end{document}